\documentclass[20pt]{article}

\usepackage{amsmath,amsfonts,amssymb,amsthm,graphics,amscd}
\usepackage{latexsym,mathrsfs,mathtools}

\usepackage{color}

\newtheorem{theorem}{Theorem}[section]
\newtheorem{lemma}[theorem]{Lemma}
\newtheorem{corollary}[theorem]{Corollary}

\newtheorem{example}[theorem]{Example}
\newtheorem{definition}[theorem]{Definition}
\newtheorem{proposition}[theorem]{Proposition}
\newtheorem{remark}[theorem]{Remark}

\def\NN{\mathbb{N}}
\def\CC{\mathbb{C}}

\newcommand\M{\mathcal{M}}

\numberwithin{equation}{section}
\newcommand{\bex}{\begin{example}\rm}
\newcommand{\eex}{\end{example}}

\def\inter{{\rm int}}

\def\acc{{\rm acc}}
\def\iso{{\rm iso}}
\def\dim{{\rm dim}}

\def\snoi{\smallskip\noindent}
\def\i{{\rm ind}}

\def\codim{{\rm codim}\, }
\def\NN{{\mathbb N}}

\def\CC{{\mathbb C}}

\def\B{\mathcal{B}}
\def\M{\mathcal{M}}

\def\D{\mathcal{D}}

\def\DDD{\bf{D}}

\def\ds{\displaystyle}
\def\p{\partial}
\def\A{\mathcal{A}}
\def\N{\mathcal{N}}
\def\R{\mathcal{R}}

\def\dis{{\rm dis}}

\def\inter{{\rm int}}

\def\acc{{\rm acc}}
\def\iso{{\rm iso}}
\def\dim{{\rm dim}}

\def\snoi{\smallskip\noindent}
\def\i{{\rm ind}}

\def\codim{{\rm codim}\, }
\def\NN{{\mathbb N}}

\def\CC{{\mathbb C}}

\def\R{\mathcal{R}}

\def\B{\mathcal{B}}
\def\M{\mathcal{M}}

\def\D{\mathcal{D}}

\def\DDD{\bf{D}}

\def\ds{\displaystyle}
\def\p{\partial}
\def\c{\'c}

\def\dis{{\rm dis}}

\begin{document}

\setcounter{footnote}{1}

\title{One-sided  generalized  Drazin-Riesz and  one-sided generalized  Drazin-meromorphic  invertible operators}

%\title{Left and right generalized  Drazin-Riesz and  generalized  Drazin-meromorphic  invertible operators
%}

\author {\ Sne\v zana  \v{C}. \v{Z}ivkovi\'{c}-Zlatanovi\'{c}\footnote{The author is
supported by the Ministry of Education, Science and Technological
Development, Republic of Serbia, grant no. 451-03-65/2024-03/200124.}}

\date{}

\maketitle
\setcounter{page}{1}

\begin{abstract}
The aim of this  paper is to  introduce and study left and right versions of the class of generalized Drazin-Riesz invertible operators, as well as left and right versions of the class of generalized Drazin-meromorphic  invertible operators.
 \end{abstract}



%\section{Introduction and definitions}

\section{Introduction and background}

In 1958 Drazin introduced a new  kind of a generalized inverse \cite{Drazin}: % in associative
%rings and semigroups.
an element $a$ of  an associative ring (or a semigroup) $\A$  is  Drazin invertible if there exists
an element $b\in\A$ such that
$$ab = ba,\ bab = b,\   a^nba = a^n,
$$
for some nonnegative integer $n$.

%The smallest such $n$ is called the Drazin index of $a$ and it is denoted by $\ind (a)$. If $a$ has a Drazin inverse, then the Drazin inverse of $a$ is unique and it is denoted by $a^d$ , \cite{K}.%, \cite{Boasso}.
%\ If $\ind(a)\le 1$, then $a$ is group invertible and $a^d$ is known as the group inverse of $a$, frequently denoted by $a^{\sharp}$ \cite{DR}.
%By $\A^{D}$ we denote the set of all Drazin invertible elements of $\A$.

The concept of  generalized Drazin invertible elements  was
introduced by  Koliha \cite{K}: an element $a$ of a Banach algebra $ \A$   is generalized Drazin
invertible  if  there exists $b\in \A$ such that
\begin{equation*}
  ab=ba,\ bab=b, \ aba-a\ {\rm is\ quasinilpotent}.
\end{equation*}
%Such an element $b$, if it exists, is unique, it is called the generalized Drazin inverse (Koliha-Drazin inverse) of $a$, and is denoted by $a^D$. By $\A^{KD}$ we denote the set of all Koliha-Drazin invertible elements of $\A$.
This notation has been further  generalized by replacing the third condition in the previous definition
 with the condition that $aba-a$ is $T$-Riesz, where $T:\A\to\B$ is a Banach algebra homomorphism satisfying the strong Riesz property, and so the notation
   of generalized Drazin-$T$-Riesz invertible elements was introduced in \cite{GDR}. %If the third condition in the previous definition is replaced
By replacing the third condition in the previous definition
   with the condition that the Drazin spectrum of $aba-a$ does not contain non-zero elements, the notation of $d$-Drazin invertible elements   was introduced in \cite{GDR}.
The concept of generalized Drazin-$T$-Riesz invertible elements  extends the concept of generalized Drazin-Riesz   invertible operators,
   while the concept of $d$-Drazin invertible elements  is an extension  of the concept of generalized Drazin-meromorphic invertible operators: in the algebra  of bounded linear operators $L(X)$ on a complex Banach space $X$  an operator $A\in L(X)$ is  generalized Drazin-Riesz
invertible, if there exists $B \in L(X)$ such that
$AB=BA,\ BAB=B,\ ABA-A \ \text{is Riesz}$ \cite{ZC}. By   replacing the third condition in the previous definitions  by the condition that $ABA-A$ is meromorphic, we get the concept of generalized Drazin-meromorphic invertible operators  \cite{ZD}.

%In the algebra  of bounded linear operators $L(X)$ on a complex Banach space $X$
Necessary and sufficient for $A\in L(X)$  to be
Drazin invertible is
  that  \cite{King}, while  $A$  is generalized
Drazin invertible if and only if it can be decomposed
into a direct sum of an invertible operator and a quasinilpotent one \cite{K}.
Recently,  Ghorbel  and  Mnif \cite{Ghorbel Mnif} introduced the concept of left (resp., right) generalized Drazin invertible operators in $L(X)$, and
 proved that this kind of operators is characterized by the property that they can be decomposed into a direct   sum of a nilpotent operator and   a left invertible  (resp., right invertible) operator. On a  Hilbert space  the class of  left (resp., right) Drazin invertible operators coincides with the class  bearing the same name
  in existing literature (see \cite{Aiena Carpintero}, \cite{Aiena}),
 though  they do not coincide on an arbitrary Banach space \cite[p. 170]{ZC}.
 %in the Banach space settings  this is not true in a general case \cite[p. 170]{ZC}.  %operators are distinguished from other classes of  operators
Berkani, Ren and Jiang introduced the concept of left (resp., right) Drazin invertible elements in  an associative ring  $R$ with a unit: an element $a\in R$ is left (resp., right) Drazin invertible if and only if there exists an element $b\in R$ such that  $bab=b^2a=b$ and $a-aba=a-ba^2$ is nilpotent (resp. $bab=ab^2=b$ and $a-aba=a-a^2b$ is nilpotent) if and only if there exists an idempotent $p$ commuting with  $a$ such that $ap$ is nilpotent and $a+p$ is left (resp., right)  invertible \cite[Theorems 2.5 and 2.6]{Berkani}, \cite[Lemma 2.1, Theorems 2.1 and 2.2]{Ren Jiang}. From \cite[Theorems 3.4 and 3.8]{Ghorbel Mnif} it follows that the set of  left (resp., right) Drazin invertible operators in $L(X)$, in the sense of Ghorbel and Mnif, coincides with the set of  left (resp., right) Drazin invertible elements   in the Banach algebra $L(X)$, in the sense of Berkani, Ren and Jiang.

In \cite{Thome}  Ferreyra,  Levis and Thome  introduced the concept  of left (resp., right) generalized Drazin invertible operators in $L(X)$ in terms of generalized Kato decomposition. They proved that this kind of operators is characterized by the property that they can be decomposed into a direct   sum of a quasinilpotent operator and   a left invertible  (resp., right invertible) operator. In the case of  Hilbert spaces the class of left (resp., right) generalized Drazin invertible operators  coincides with the class bearing the same name and introduced by   Hocine,  Benharrat and Messirdi in \cite{algeria}.
  These classes do not coincide on an arbitrary Banach space. In \cite[Propositions 2.1 and 2.2]{Ounandjela} Ounandjela et al.  proved that an operator $A\in L(X)$, where $X$ is a Hilbert space, is left (resp., right) generalized  Drazin invertible if and only if there exists $B\in L(X)$ such that $ABA=BA^2,\ B^2A=B=BAB, \ A-ABA \ {\rm is\ quasinilpotent}$ (resp., $ABA=A^2B,\ AB^2=B=BAB, \ A-ABA \ {\rm is\ quasinilpotent}$).

  The purpose of this paper is to introduce the concept of one-sided  generalized Drazin-Riesz and one-sided  generalized Drazin-meromorphic  invertibility in the algebra $L(X)$: an operator $A\in L(X)$  is   left  (resp., right) generalized Drazin-Riesz invertible if there is $B\in L(X)$ such that
$ABA=BA^2,\ B^2A=B, \ A-ABA \ {\rm is\ Riesz}$ (resp., $ABA=A^2B,\ AB^2=B, \ A-ABA \ {\rm is\ Riesz}$); if the third condition in the previous definition is replaced by the condition that $A-ABA \ {\rm is\ meromorphic}$ then we get the concept of left (resp., right) generalized Drazin-meromorphic invertible operators.
   Starting from these algebraic definitions we obtain various characterizations of left (resp., right) generalized Drazin-Riesz invertible and left (resp., right) generalized Drazin-meromorphic  invertible operators, among others that an operator $A\in L(X)$ is left  (resp., right) generalized Drazin-Riesz invertible if and only if it can be decomposed in a direct sum of a left (resp., right) invertible operator and a Riesz operator, while $A$ is left  (resp., right) generalized Drazin-meromorphic invertible if and only if it can be decomposed in a direct sum of a left (resp., right) invertible operator and a meromorphic operator.
%By  introducing  the concept of generalized Saphar-Riesz decomposition as well as the concept of generalized Saphar-meromorphic decomposition we obtain some properties  of left (right) generalized Drazin-Riesz invertible and left (right) generalized Drazin-merpmorphic  invertible operators, which
% provides us to get  various properties of corresponding spectra.
The concept of left (resp., right) Drazin invertible operators appears in a characterization of   left  (resp., right) generalized Drazin-meromorphic invertible oprators: an operator $A\in L(X)$ is left  (resp., right) generalized Drazin-meromorphic  invertible if and only if there exists a projector $P\in L(X)$   commuting with $A$
such that $A+P$ is left (resp., right) Drazin invertible and $AP$ is meromorphic.
The introduction of the notation  of generalized Saphar-Riesz decomposition as well as the notation of generalized Saphar-meromorphic decomposition allowed us to obtain some  characteristics of these classes of
  operators, which results in obtaining various properties of the corresponding spectra. 
   %By using the result that if an upper (resp., lower) Drazin invertible operator is not of Saphar type, then it does not admit a  generalized Saphar-meromorphic decomposition \cite{GZ},
   By using the concept of generalized Saphar-meromorphic decomposition
    it is proved  that
     on an arbitrary Banach space  the class of left (resp., right) generalized Drazin-Riesz invertible operators does not coincide with the class of operators that are characterized by the property that they can be decomposed into a direct sum of a bounded bellow (resp.,  surjective) operator and a Riesz operator. %\cite[Theorems 2.4 and 2.5]{ZC}.
  Also it is proved  that on an arbitrary Banach space  the class of left (resp., right) generalized Drazin-meromorphic  invertible operators does not coincide with the class of operators that are characterized by the property that they can be decomposed into a direct sum of a bounded bellow (resp.,  surjective) operator and a meromorphic operator (Remark \ref{nakon}). % \cite[Theorems 6 and 7]{ZD}.

The paper is organized as follows. The second section  contains some definitions and the third section contains some   preliminary results.    In  the forth section we introduce and investigate   left (right)
generalized Drazin-Riesz invertible operators. The fifth  section is devoted to
left (right) generalized Drazin-merpmorphic  invertible operators. The sixth section contains  various properties of corresponding spectra.

\section{Definitions}% and preliminary results}

 Let  $L(X)$   be  the Banach algebra of bounded linear operators acting on  an infinite dimensional complex  Banach space $X$. The dual of $X$ is denoted by  $X^\prime$, and the dual of $A\in L(X)$ by $A^\prime$.   Here $\mathbb{N} \, (\mathbb{N}_0)$ denotes the set of all positive
(non-negative) integers, $\mathbb{C}$ denotes the set of all
complex numbers.

 Let $A\in L(X)$.  By  $\sigma(A)$, $\sigma_{l}(A)$, $\sigma_{r}(A)$ and $\sigma_p(A)$ we denote  its  spectrum, left spectrum, right spectrum and   point spectrum, respectively.  The compression spectrum of $A$, denoted by $\sigma_{cp}(A)$, is the set of all complex $\lambda$ such that $A-\lambda I$ does not have dense range.
We use $\N(A)$ and $\R(A)$, respectively, to
denote the null-space and the range of $A$. %, and $\asc(A)$ and $\dsc(A)$, respectively,  to denote its ascent and descent.

For   $n\in\NN_0$ set
$\alpha_n(A)=\dim (\N(A)\cap \R(A^n))$
and
   $\beta_n(A)=\codim (\R(A)+\N(A^n))$, and set  $\alpha(A)=\alpha_0(A)$ and  $\beta(A)=\beta_0(A)$.
  We define the infimum of the empty set to be  $\infty$.
The    ascent $ a(A) $   and the descent $d(A)$  of $A$ are defined by
 $ a(A)=\inf \{ n\in\NN_0:\alpha_{n}(A)=0 \}$ %= \inf \{n\in\NN_0 : \N(A^{n})=\N(A^{n+1})\}$
and
 $ d
 (A)=\inf \{ n\in\NN_0:\beta_{n}(A)=0 \}$. %=\inf  \{n\in\NN_0: \R(A^n) = \R(A^{n+1})\}.$
The  essential ascent $ a_e(A) $  and  essential  descent $d_e(A)$  of $A$ are defined by
 $ a_e(A)=\inf \{ n\in\NN_0:\alpha_{n}(A)<\infty \}$
\ and
 $
  d_e(A)=\inf \{ n\in\NN_0:\beta_{n}(A)<\infty \}$.
 The descent spectrum of $A$ is
$     \sigma_{dsc} (A)=\{\lambda\in\CC:d(T-\lambda I)=\infty\}.$
The  degree of stable iteration $\dis(A)$  is defined by:
$$
\dis(A)=\inf\{n\in\NN_0:m\ge n, m\in\NN\Longrightarrow \N(A)\cap \R(A^n)= \N(A)\cap \R(A^m)\}.
$$

A  subspace $M$ of $X$ is    complemented in $X$ if it is closed and  there is a closed subspace $N$ of $X$ such that $X=M\oplus N$.
If $\alpha(A)<\infty$, $a(A)<\infty$ and $\R(A)$ is complemented, then $A$ is called {\it left Browder}. If $\beta(A)<\infty$, $d(A)<\infty$ and $\N(A)$ is complemented, then $A$ is called {\it right Browder}.
 The set of
left (right) Browder operators is denoted by $\mathcal{B}_l(X)$ $(\mathcal{B}_r(X))$, while  $\mathcal{B}(X)=\mathcal{B}_l(X)
\cap \mathcal{B}_r(X)$ is the set of Browder operators.
It is well known that  $A $ is    Drazin invertible if and only if   $a(A) < \infty $ and $d(A)<\infty$.
 $A\in L(X)$ is called {\it left  Drazin invertible} if $a(A)<\infty$ and  $\R(A)+ \N(A^{a(A)})$ is a complemented subspace of $X$. If $d(A)<\infty$ and  $\N(A)\cap \R(A^{d(A)})$ is  complemented in $X$, then  $A$ is called {\it right   Drazin invertible} \cite{Ghorbel Mnif}.
  An operator   $A\in L(X)$ is {\it upper   Drazin invertible} operator if
 $
 a(A) < \infty$ and $ \R(A^{a(A)+1})$  is closed.  If $d(A) < \infty$  and $ \R(A^{d(A)})$  is closed, then $A$ is called {\it lower    Drazin invertible}.
The   Browder  spectrum, the left Browder  spectrum,  the right Browder  spectrum, the  Drazin   spectrum, the left Drazin  spectrum,  the right Drazin  spectrum and the upper Drazin  spectrum  of $A$ are denoted by $\sigma_{\B}(A)$, $\sigma_{\B}^l(A)$,
$\sigma_{\B}^r(A)$, $\sigma_{D}(A)$, $\sigma_{D}^l(A)$,
$\sigma_{D}^r(A)$ and  $\sigma_{D}^+(A)$,
 respectively.

If $K \subset \mathbb{C}$, then $\partial K$ is the
boundary of $K$, $ \operatorname{acc}  K$ is the set of accumulation points of
$K$,  $\inter\, K$ is the set of interior points of $K$ and $\iso\, K$ is the set of isolated points of $K$.
The {\it connected hull}  of a compact set $K\subset\CC$, denoted by $\eta K$, is the complement of the unbounded
component of $\CC\setminus K$ \cite[Definition
7.10.1]{H3}.% Given a compact subset $K$ of the plane, a hole of $K$
%is a bounded component of $\CC\setminus K$, and so a hole of $K$ is
%a component of $\eta K\setminus K$.
We recall  that, for compact subsets $H,K\subset\CC$, the following implication holds (\cite[Theorem
7.10.3]{H3}):
\begin{equation}\label{spec.1}
\partial H\subset K\subset
H\Longrightarrow\partial H\subset\partial K\subset K\subset
H\subset\eta K=\eta H\ .
\end{equation}

An operator $A\in L(X)$ has the single-valued extension property at $\lambda_0\in\CC$ (SVEP at $\lambda_0$ for breviety) if for every open disc $D_{\lambda_0}$ centered at $\lambda_0$ the only analytic function $f:D_{\lambda_0}\to X$ satisfying $(A-\lambda)f(\lambda)=0 $ for all $\lambda\in \D_{\lambda_0}$ is the function $f\equiv 0$.
An operator $A\in L(X)$ is said to have the SVEP if $A$ has the SVEP at every point $\lambda\in\CC$.

An operator $A\in L(X)$ is  a {\em Riesz operator}, if $\alpha(A-\lambda I)<\infty$ and $\beta(A-\lambda I)<\infty$
 for every non-zero $\lambda \in \mathbb{C}$.
   An operator $A\in L(X)$   is called  {\it meromorphic} if its non-zero spectral points are poles of its resolvent, and in that case we shall write $A\in(\M)$. %It is well known that $A$ is Riesz operator if and only if
% every nonzero point of  $\sigma (A)$ is a pole of the finite
%algebraic multiplicity.
Every Riesz operator is meromorphic (see \cite[Corollary 3.1]{SANU}).
 We say that $A$ is {\it polinomially  Riesz} ({\it polinomially  meromorphic}) if there exists non-trivial polynomial $p$ such that $p(A)$ is Riesz (meromorphic).

 If $M$ is an $A$-invariant subspace of $X$, we define $A_M:M
\to M$ as $A_Mx=Ax, \, x \in M$.
If $M$ and $N$ are two closed
$A$-invariant subspaces of $X$ such that $X=M \oplus N$, we say that
$A$ is {\em completely reduced} by the pair $(M,N)$ denoting  it by $(M,N) \in Red(A)$.  In this case we write $A=A_M \oplus
A_N$.

%If for an operator $A\in L(X)$ there exists a pair $(M,N)\in Red(A)$ such that $A_M$ is bounded bellow (resp. surjective) and $A_N$ is Riesz, we say that $A$ is  {\em upper  generalized Drazin-Riesz invertible} (resp., {\it lower generalized Drazin-Riesz invertible}) (see \cite[Theorems 2.4 and 2.5]{ZC}).  Furthermore, if there exists a pair $(M,N)\in Red(A)$ such that $A_M$ is bounded bellow (resp. surjective) and $A_N$ is meromorphic, then we will say that $A$ is  {\em upper  generalized Drazin-meromorphic invertible} (resp., {\it lower generalized Drazin-meromorphic  invertible}) (see \cite[Theorems 6 and 7]{ZD}).

An operator   $A\in L(X)$  is said to be {\em Kato} if $\R(A)$ is closed and $\N(A) \subset
\R(A^n)$ for every $ n \in \mathbb{N}$.
An operator $T \in L(X)$ is said to admit a {\em generalized Kato-Riesz
decomposition}, abbreviated as $GKRD$, if there exists a pair $(M,N)
\in Red(T)$ such that $T_M$ is Kato and $T_N$ is Riesz.
An operator $A \in L(X)$ is said to  admit a {\em generalized Kato-meromorphic
decomposition}, abbreviated to $GK(\M)D$, if there exists a pair $(M,N)
\in Red(T)$ such that $T_M$ is Kato and $T_N$ is meromorphic. The generalized Kato-Riesz spectrum and the generalized Kato-meromorphic spectrum of $A\in L(X)$ are denoted by $\sigma_{gKR}(A)$ and $\sigma_{gK\M}(A)$, respectively (see \cite{ZC}, \cite{ZD}).

% If $A$ is a relatively regular Kato operator, then $A$ is called  a {\it Saphar operator}.
  If $A$ is Kato operator and $\N(A)$ and $\R(A)$ are complemented in $X$, then $A$ is called  a {\it Saphar operator}.
% An operator $T \in
%L(X)$ is {\it essentially Saphar} if
%$\N(A)\stackrel{\rm e}{\subset} \cap_{n=1}^{\infty} \R(A^n)$ and  $A$ is relatively regular \cite[p. 233]{Mu}.
An operator   $A\in L(X)$ is called  {\it essentially Saphar} if
$\N(A)\stackrel{\rm e}{\subset} \cap_{n=1}^{\infty} \R(A^n)$, and    $\N(A)$ and $\R(A)$ are complemented in $X$ \cite[p. 233]{Mu}.
 From  \cite[Theorem 2.1]{voki} and \cite[Lemma 2.1]{IEOT}  it follows that  $A$ is essentially Saphar if and only if there exists $(M,N)\in Red(A)$ such that $\dim N<\infty$, $A_N$ is nilpotent and  $ A_M$ is Saphar.
 The degree of a nilpotent operator $A$ is the smallest $d\in\NN_0$ such that $A^d=0$.
  We say that $A\in L(X)$ is {\it of Saphar type  of degree} $d$ if there exists a pair $(M,N) \in Red(A)$ such that
$A_M$ is Saphar and $A_N$ is nilpotent of degree $d$ \cite{ZS}. In that case $\dis(A)=d$.
We recall that $A\in L(X)$ is of Saphar type  of degree $d$ if and only if   $\dis (A)=d$ is finite and the subspaces $\R(A)+\N(A^d)$ and $\N(A)\cap \R(A^{d})$  are complemented \cite[Theorem 4.2]{ZS}. It is said that $A\in L(X)$  admits a {\em generalized Saphar decomposition} if  there exists a pair $(M,N)\in Red(A)$ such that $A_M$ is Saphar and $A_N$ is quasinilpotent \cite{MZ}. The   Saphar  spectrum, the essentially Saphar  spectrum,  the Saphar type  spectrum and the  generalized Saphar  spectrum of $A$ are denoted by $\sigma_{S}(A)$, $\sigma_{eS}(A)$,
$\sigma_{St}(A)$ and  $\sigma_{gS}(A)$,
 respectively.

\section{Preiliminary results}

\begin{lemma}\cite[Lemma 3.2]{ZS}\label{zatvorenost} Let $X=X_1\oplus X_2\oplus\dots\oplus X_n$ where $X_1,\ X_2,\dots ,X_n$ are closed subspaces of $X$ and let $M_i$ be a  subset of $X_i$, $i=1,\dots,n$. Then the set \break $M_1+ M_2+\dots+ M_n$ is   closed if and only if $M_i$ is closed for each $i\in\{1,\dots,n\}$.
\end{lemma}

\begin{lemma}\cite[Lemma 3.3]{ZS} \label{glavna lema} Let $X=X_1\oplus X_2\oplus \dots\oplus X_n$ where $X_1,\ X_2,\dots ,X_n$ are closed subspaces of $X$ and let $M_i$ be a  subspace of $X_i$, $i=1,\dots,n$. Then the subspace   $M_1\oplus M_2\oplus\dots\oplus M_n$ is a  complemented subspace of $X$ if and only if $M_i$ is a complemented subspace of $X_i$ for each $i\in\{1,\dots,n\}$.
\end{lemma}

\begin{lemma}\label{bot}\cite[Lemma 2.6]{MZ}  If $M$ is a  complemented subspace of $X$, then $M^{\bot}$ is a  complemented subspace of $X^\prime$.
\end{lemma}

\begin{lemma}\label{complemented in}\cite[Lemma 3.5]{ZS} If  $M$ is complemented subspace of $X$ and  $M_1$ is  a closed subspace of $X$ such that $M\subset M_1$, then $M$ is complemented in $M_1$.
\end{lemma}

\begin{lemma}\cite[Lemma 3.11]{ZS} \label{suma left} For  $A\in L(X)$  let there exists a pair $(M,N) \in Red(A)$.
Then   $A$ is left (resp., right) invertible
  if and only if $A_M$ and $A_N$ are left (resp., right) invertible.
\end{lemma}
\begin{proof} From the equalities $\N(A)=\N(A_M)\oplus\N(A_N)$, $\R(A)=\R(A_M)\oplus\R(A_N)$ and Lemma \ref{glavna lema} it follows that $\N(A)$  (resp. $\R(A)$) is complemented if and only if $\N(A_M)$ and $\N(A_N)$  (resp. $\R(A_M)$ and $\R(A_N)$) are complemented subspaces in $M$ and $N$, respectively. Since a bounded linear  operator is left (resp. right) invertible if and only if it is injective (resp., surjective) with a complemented range (resp. null-space), we get the desired conclusion.
\end{proof}

\begin{lemma}\label{Riesz}{}
For  $A\in L(X)$  let there exists a pair $(M,N)\in Red (A)$. Then
 
 \snoi {\rm (i)} \cite[Lemma 2.11]{ZDHB} 
  $A$ is
Riesz  if and only if $A_M$ and $A_N$ are Riesz.

\snoi {\rm (ii)} \cite[Lemma 2]{ZD} $A$ is meromorphic if and only if $A_M$ and $A_N$ are meromorphic.

\snoi {\rm (iii)} \cite[Lemma 3.11]{ZS}   $A$ is Saphar  if and only if $A_M$ and $A_N$ are Saphar.

\snoi {\rm (iv)}    \cite[Lemma 2.1]{IEOT}  $A$ is
left (resp., right) Browder  if and only if $A_M$ and $A_N$ are left (resp., right) Browder.
\end{lemma}

%\begin{lemma}\label{lema-meromorphic}\cite[Lemma 2]{ZD}
%For  $A\in L(X)$  let there exists a pair $(M,N)\in Red (A)$. Then
%  $A$ is meromorphic if and only if $A_M$ and $A_N$ are meromorphic.
%\end{lemma}

%\begin{lemma}\cite[Lemma 2.1]{IEOT} \label{Browder}
%Let $A\in L(X)$ and let $(M,N)\in Red (A)$. Then
%  $A$ is
%left (right) Browder  if and only if $A_M$ and $A_N$ are left (right) Browder.
%\end{lemma}
%
% \begin{lemma}\cite[Lemma 3.11]{ZS} \label{suma Safarovih} For  $A\in L(X)$  let there exists a pair $(M,N) \in Red(A)$.
%Then   $A$ is Saphar  if and only if $A_M$ and $A_N$ are Saphar.
%\end{lemma}

\begin{lemma}\label{Saphar type}  For  $A\in L(X)$  let there exists a pair $(M,N) \in Red(A)$.
Then   $A$ is of Saphar type  if and only if $A_M$ and $A_N$ are of Saphar type.
\end{lemma}
\begin{proof} For $n\in\NN_0$ it holds

\parbox{11cm}{\begin{eqnarray*}
\N(A)\cap \R(A^n)& =&(\N(A_M)\cap \R(A_M^n))\oplus (\N(A_N))\cap  \R(A^n_N)),\\ \R(A)+ \N(A^n)& =&(\R(A_M)+ \N(A_M^n))\oplus (\R(A_N))+  \N(A^n_N)).
\label{degree}
\end{eqnarray*}}  \hfill
\parbox{1cm}{\begin{eqnarray}\end{eqnarray}}

\noindent which implies that  $\dis(A)<\infty$ if and only if $\dis(A_M)<\infty$ and $\dis(A_N)<\infty$, and in that case $\dis(A)=\max\{\dis(A_M),\dis(A_N)\}$.
Let $d=\dis(A)<\infty$. From \eqref{degree} and Lemma \ref{glavna lema} it follows that
 $\N(A)\cap \R(A^d)$ and $\R(A)+ \N(A^d)$ are complemented if and only if $\N(A_M)\cap \R(A_M^d)$ and $\R(A_M)+ \N(A_M^d)$ are complemented subspaces of $ M$,  and $\N(A_N)\cap \R(A_N^d)$ and $\R(A_N)+ \N(A_N^d)$ are complemented subspaces of $ N$. Now according  to \cite[Theorem 4.2]{ZS} we conclude that  $A$ is of Saphar type  if and only if $A_M$ and $A_N$ are of Saphar type.
\end{proof}

%In order to characterize left (right) generalized  Drazin-meromorphic invertible operators we will  need the following lemmas.

\begin{lemma}\label{left decom}  For  $A\in L(X)$  let there exists a pair $(M,N) \in Red(A)$.
Then   $A$ is left (right) Drazin invertible  if and only if $A_M$ and $A_N$ are left (right) Drazin invertible.
\end{lemma}
\begin{proof} From \cite[Corollary 4.23]{ZS} we have that $A$ is left  Drazin invertible  if and only if $A$ is of Saphar type and $a(A)<\infty$.
As $a(A)<\infty$ if and only if $a(A_M)<\infty$ and $a(A_N)<\infty$, by using Lemma \ref{Saphar type} and again \cite[Corollary 4.23]{ZS} we conclude  $A$ is left  Drazin invertible  if and only if $A_M$ and $A_N$ are left  Drazin invertible.

The assertion for the case of right  Drazin invertible operators is proved in a similar way by using \cite[Corollary 4.24]{ZS}.
\end{proof}

\begin{lemma} \label{left Drazin} Let $A,B\in L(X)$. Then:

 \snoi {\rm (i)} If $A$ and $B$ are   left Drazin invertible  and  $P\in L(X)$ is a projector commuting with $A$ and $B$, then $AP+B(I-P)$ is  left Drazin invertible.

  \snoi {\rm (ii)} If $A$ and $B$ are   right Drazin invertible  and  $P\in L(X)$ is a projector commuting with $A$ and $B$, then $AP+B(I-P)$ is   right  Drazin invertible.

\end{lemma}
\begin{proof} (i): Suppose that $A$ and $B$ are  left  Drazin invertible, and that $P\in L(X)$ is a projector commuting with $A$ and $B$. Since $(\R(P), \N(P))\in Red(A)$ and $(\R(P), \N(P))\in Red(B)$, according to Lemma \ref{left decom} it follows that $A_{\R(P)}$ and $B_{\N(P)}$ are  left  Drazin invertible. As  $(\N(P), \R(P))\in Red(AP+B(I-P))$, we have that
\begin{equation}\label{sd}
  AP+B(I-P)=(AP+B(I-P))_{\R(P)}\oplus (AP+B(I-P))_{\N(P)}=A_{\R(P)}\oplus B_{\N(P)}.
\end{equation}
  Again using  Lemma \ref{left Drazin}, from \eqref{sd}  we conclude that $AP+B(I-P)$ is  left Drazin invertible.

%Suppose that $A$ and $B$ are  left  Drazin invertible, and that $P\in L(X)$ is a projector commuting with $A$ and $B$. From \cite[Corollary 4.23]{ZS} it follows that $A$ and $B$ are of Saphar type and almost left invertible.  From the proof of Lemma \ref{almost+Fredholm} we have that $AP+B(I-P)$ is almost left invertible. Since $(\R(P), \N(P))\in Red(A)$ and $(\R(P), \N(P))\in Red(B)$, according Lemma \ref{Saphar type} it follows that $A_{\R(P)}$ and $B_{\N(P)}$ are of Saphar type. As  $(\N(P), \R(P))\in Red(AP+B(I-P))$, we have that
%$$
%AP+B(I-P)=(AP+B(I-P))_{\R(P)}\oplus (AP+B(I-P))_{\N(P)}=A_{\R(P)}\oplus B_{\N(P)},
%$$
%which is an operator of Saphar  type according Lemma \ref{Saphar type}. Using again \cite[Corollary 4.23]{ZS}  we get that $AP+B(I-P)$ is  left Drazin invertible.

(ii) It is proved in an analogous  way as  (i).
\end{proof}

\begin{lemma} \label{almost+Fredholm} Let $A,B\in L(X)$. Then:

 \snoi {\rm (i)} If $A$ and $B$ are   left Browder and  $P\in L(X)$ is a projector commuting with $A$ and $B$, then $AP+B(I-P)$ is  left  Browder.

  \snoi {\rm (ii)} If $A$ and $B$ are   right Browder and  $P\in L(X)$ is a projector commuting with $A$ and $B$, then $AP+B(I-P)$ is   right  Browder.

\end{lemma}
\begin{proof} The proof  can be demonstrated by using Lemma \ref{Riesz} (iv) in a similar way as  the proof of Lemma \ref{left Drazin}.
 \end{proof}

\section{Left and right generalized  Drazin-Riesz invertible operators}\label{left Riesz section}

 The aim of this  section is to  introduce and explore the concept of generalized Saphar-Riesz decomposion and the concept of left and right generalized  Drazin-Riesz invertible operators.

\begin{definition} \rm Let $A\in L(X)$. If  there exists a pair $(M,N)
\in Red(A)$ such that $A_M$ is Saphar and $A_N$ is Riesz, then $A$  admits  a {\em generalized Saphar-Riesz
decomposition}, or shortly $A$ admits a GSRD$(M,N)$ (or   $A$ admits a GSRD).
% An operator $A \in L(X)$ is said to  admit a {\em generalized Saphar-Riesz
%decomposition},  if there exists a pair $(M,N)
%\in Red(A)$ such that $A_M$ is Kato and $A_N$ is Riesz. In that case we shall  say that $A$ admits a GSRD$(M,N)$, or shortly a GSRD.
\end{definition}

\begin{theorem}\label{prva S} Let $A\in L(X)$. The following conditions are equivalent:

\snoi {\rm (i)}  $A$ admits a GSRD;

\snoi {\rm (ii)} There exists a projector $P\in L(X)$   commuting with $A$
such that $A+P$ is essentially  Saphar   and $AP$ is Riesz;

\snoi {\rm (iii)} There exists a projector $P\in L(X)$   commuting with $A$
such that $A+P$ is of Saphar type  and $AP$ is Riesz.
\end{theorem}
\begin{proof} (i)$\Longrightarrow$(ii): Let $A$ admit a GSRD.  Then there exists $(M,N)\in Red(A)$ such that $ A_M$ is Saphar and $ A_N$ is Riesz. Let $P\in L(X)$ be a projector such that $\R(P)=N$ and $\N(P)=M$. Then  $P$ commutes with $A$, and hence  $(M,N)\in Red (AP)$, $(M,N)\in Red (A+P)$. Thus
$ AP=(AP)_M\oplus (AP)_N=0\oplus A_N$, and since
 $A_N$ is Riesz,  according to Lemma \ref{Riesz} (i) we obtain that $AP$ is Riesz, while from \cite[Theorem 3.2]{SANU}  we conclude that
  $(A+P)_N=A_N+I_N$ is Browder. From \cite[Theorem 2.32]{SANU} it follows that there exists closed subspaces $N_1$ and $N_2$ of $N$ such that $(N_1,N_2)\in Red((A+P)_N)$, $\dim N_2<\infty$, $(A+P)_{N_1}$ is invertible and $(A+P)_{N_2}$ is nilpotent. According to  Lemma
  \ref{zatvorenost} we have that the subspace $M\oplus N_1$ is   closed.
   As $M$ and $N_1$ are  $A+P$-invariant, we have that  $M\oplus N_1$ is   $A+P$-invariant. Since $M\oplus N_1$  is complemented in $X$  with $N_2$, we have that $(M\oplus N_1,N_2)\in Red(A+P)$. As $(A+P)_M=A_M$ and $(A+P)_{N_1}$    are Saphar, Lemma
  \ref{Riesz} (iii) ensures that $(A+P)_{M\oplus N_1}$ is Saphar. Consequently,  $A+P$ is essentially  Saphar.

  (ii)$\Longrightarrow$(iii): It is clear.

  (iii)$\Longrightarrow$(i): Let there exist a projector $P\in L(X)$   commuting with $A$
such that $A+P$ is of Saphar type  and $AP$ is Riesz. Then $(\R(P), \N(P))\in Red(AP)$, $(\R(P), \N(P))\in Red(A+P)$. From  Lemma \ref{Riesz} (i) it follows that $A_{\R(P)}=(AP)_{\R(P)}$ is Riesz, while from Lemma \ref{Saphar type}  it follows that $A_{\N(P)}=(A+P)_{\N(P)}$ is of Saphar type. Hence there exist closed subspaces $N_1,N_2$ of $\N(P)$ such that $(N_1,N_2)\in Red(A_{\N(P)})$, $A_{N_1}$ is Saphar and $A_{N_2}$ is nilpotent. According to Lemma
  \ref{zatvorenost} the subspace $N_2\oplus \R(P)$ is closed. Since $(N_1,N_2\oplus \R(P))\in Red(A)$ and $A_{N_2\oplus \R(P)}=A_{N_2}\oplus A_{\R(P)}$ is Riesz according to Lemma \ref{Riesz} (i), we obtain that $A$%=A_{N_1}\oplus A_{N_2\oplus \R(P)}$
  \  admits a GSRD$(N_1, N_2\oplus \R(P))$.
\end{proof}

\begin{definition}\label{left-R-Drazin} \rm An operator
$A\in L(X)$ is  {\it left  generalized Drazin-Riesz invertible} if there is $B\in L(X)$ such that
\begin{equation}\label{prvi uslovi}
ABA=BA^2,\ B^2A=B, \ A-ABA \ {\rm is\ Riesz}.
\end{equation}

An operator
$A\in L(X)$ is  {\it right  generalized Drazin-Riesz invertible} if there is $B\in L(X)$ such that
\begin{equation}\label{prvi uslovi r}
ABA=A^2B,\ AB^2=B, \ A-ABA \ {\rm is\ Riesz}.
\end{equation}
\end{definition}

 %Every left (right) invertible operators is left (right) generalized Drazin-Riesz invertible.
%If  $P\in L(X)$ is a projector, then it is left (right)  generalized Drazin-Riesz invertible thereby $P$ is a left (right) generalized Drazin-Riesz inverse of $P$.  %$\sigma_{\R}(0)\subset\sigma(0) =\{0\}$
%Also, if $A$ is Riesz, then it is left (right) generalized Drazin-Riesz invertible and  $0$ is a left (right) generalized Drazin-Riesz inverse of $A$.

In \cite[Definition 7.5.2]{H3} Harte introduced the notation  of a quasipolar element in a Banach algebra:  an element $a$ of a Banach algebra $\A$ is {\it quasipolar} if there is an idempotent $q\in \A$  such that $aq=qa$,  $a(1-q)$ is quasinilpotent and $q\in (\A a)\cap(a\A)$.
 The notation  of  Risz-quasipolar operators was introduced in \cite{ZC}: an operator $A\in L(X)$  is {\it Riesz-quasipolar} if there exists a bounded projection $Q$ commuting with $A$ and satisfying
$A(I-Q)\ {\rm is\ Riesz},\ Q\in(L(X)A)\cap (AL(X))$.

\begin{definition} \rm An operator  $A\in L(X)$ is {\it left Risz-quasipolar} if there exists a projector   $Q\in L(X)$ satisfying
\begin{equation}\label{lR-qp}
 AQ=QA,\ A(I-Q)\ {\rm is\ Riesz},\ Q\in L(X)A.
\end{equation}
An operator $A\in L(X)$ is {\it right Riesz-quasipolar} if there exists a projector $Q\in L(X)$ satisfying
\begin{equation}\label{rR-qp}
 AQ=QA,\ A(I-Q)\ {\rm is\ Riesz },\ Q\in AL(X).
\end{equation}
\end{definition}

In the following theorem we give some characterizations of left generalized  Drazin-Riesz invertible operators.

\begin{theorem} \label{prva Riesz} For $A\in L(X)$ the following conditions are equivalent:

\snoi {\rm (i)}
 $A$ is left generalized Drazin-Riesz invertible;

\snoi {\rm (ii)}  There exists $C\in L(X)$ such that
\begin{equation*}%\label{drugi uslovi}
    ACA=CA^2,\ C^2A=C=CAC, \ A-ACA \ {\rm is\ Riesz};
\end{equation*}

\snoi {\rm (iii)} There exists $(M,N)\in Red(A)$ such that $ A_M$ is left invertible   and $ A_N$ is Riesz;

\snoi {\rm (iv)}
There exists a projector $P\in L(X)$   commuting with $A$
such that $A+P$ is left Browder and $AP$ is Riesz;

\snoi {\rm (v)} There exists a projector $P\in L(X)$   commuting with $A$
such that $A(I-P)+P$ is left Browder and $AP$ is Riesz;

\snoi {\rm (vi)}  There exist $C,D,P\in L(X)$ such that $C$ is  left Browder, $D$ is Riesz, $P$ is a projector  commuting with $C$ and $D$,  and  $A=C(I-P)+DP$;

\snoi {\rm (vii)}
$A$ is left Riesz-quasipolar;

\snoi {\rm (viii)} $A$ admits a GSRD  and
$0\notin {\rm int}\, \sigma_{l}(A)$;

\snoi {\rm (ix)} $A$ admits a GSRD and $A$ has the SVEP at $0$;

\snoi {\rm (x)} $A$ admits a GSRD  and
$0\notin {\rm acc}\, \sigma_{\B}^l(A)$;

\snoi {\rm (xi)} $A$ admits a GSRD  and
$0\notin {\rm int}\, \sigma_{\B}^l(A)$;

\snoi {\rm (xii)} $A$ admits a GSRD  and
$0\notin {\rm acc}\, \sigma_{D}^+(A)$;

\snoi {\rm (xiii)} $A$ admits a GSRD  and
$0\notin {\rm int}\, \sigma_{D}^+(A)$;

%\snoi {\rm (xiv)} $A$ admits a GSRD  and
%$0\notin {\rm acc}\, \sigma_{p}(A)$;

\snoi {\rm (xiv)} $A$ admits a GSRD  and
$0\notin {\rm int}\, \sigma_{p}(A)$.

%?? \snoi {\rm (ix)} There exists a projector $P\in L(X)$
%which commutes with $A$ such that $(I-P)A(I-P)$ is left invertible in $(I-P)L(X)(I-P)$ and $AP$ is Riesz  in $PL(X)P$.

\end{theorem}

\begin{proof} (i)$\Longrightarrow$(ii): Let $A$ be left  generalized Drazin-Riesz invertible. Then there exists $B\in L(X)$ such that the conditions \eqref{prvi uslovi} hold.
Set $C=BAB$. Then % It is easy to verify that $ACA=CA^2$, $C^2A=C$, $CAC=C$ and $A-ACA=A-ABA$ is Riesz.
\begin{eqnarray*}
  &&ACA=ABABA=BA^2BA=BABA^2=CA^2,\\ && C^2A=(BAB)^2A=BA(BBA)BA=BABBA=BAB=C,  \\
  && CAC=BABABAB=BABBA^2B=BABAB=B^2A^2B=BAB=C, \\
 && A-ACA=A-ABABA=A-AB^2A^2=A-ABA\ {\rm is\ Riesz}.
\end{eqnarray*}
(ii)$\Longrightarrow$(i): It is obvious.

\smallskip

(ii)$\Longrightarrow$(iii): Suppose that there exists $C\in L(X)$ and $R\in L(X)$ which is Riesz such that  $ ACA=CA^2,\ C^2A=C=CAC, \ A-ACA=R$.
Then $(CA)^2=CACA=CA$, $A(CA)=CA^2=(CA)A$, $C(CA)=C=(CA)C$, and so $CA$ is a bounded projector which commutes with $A$,  $C$ and  $R$. It implies that  $(\R(CA), \N(CA))\in Red(A)$, $(\R(CA), \N(CA))\in Red(C)$ and  $(\R(CA), \N(CA))\in Red(R)$. According to Lemma \ref{Riesz} (i) we conclude that $R_{\N(CA)}$ is Riesz. %Since $\N(CA)=\N(ACA)=\N(A-R)$, it follows that
Further, $A_{\N(CA)}=(A(I-CA))_{\N(CA)}=R_{\N(CA)}$ and so $A_{\N(CA)}$ is Riesz. As $C_{\R(CA)}A_{\R(CA)}= (CA)_{\R(CA)}=I_{\R(CA)}$, it follows that $A_{\R(CA)}$ is left invertible.

\smallskip

(iii)$\Longrightarrow$(i): Suppose that there exists $(M,N)\in Red(A)$ such that $ A_M$ is left invertible   and $ A_N$ is Riesz. Let $B_M$ is a left inverse of $A_M$, i.e. $B_MA_M=I_M$. Set $B=B_M\oplus 0_N$. Then
$$
B^2A=(B_M\oplus 0_N)^2(A_M\oplus A_N)=B_M^2A_M\oplus 0_N=B_M\oplus 0_N=B,
$$
$$ABA=(A_M\oplus A_N)(B_M\oplus 0_N)(A_M\oplus A_N)=A_MB_MA_M\oplus 0_N=A_M\oplus 0_N,$$
$$
BA^2=(B_M\oplus 0_N)(A_M\oplus A_N)^2=B_MA_M^2\oplus 0_N=A_M\oplus 0_N,
$$ and hence $ABA=B^2A$ and $A-ABA=0_M\oplus A_N$ is Riesz according to Lemma \ref{Riesz} (i). Therefore, $A$ is left generalized Drazin-Riesz invertible.

\smallskip

(iii)$\Longrightarrow$(iv): Suppose that  there exists $(M,N)\in Red(A)$ such that $ A_M$ is left invertible   and $ A_N$ is Riesz.  Let $P$ be the projector  of $X$ onto $N$ along $M$. Then  $P$ commutes with $A$, and hence $P$ commutes with $AP$ and $A+P$. Thus $(M,N)\in Red (AP)$ and $(M,N)\in Red (A+P)$. From
$$
AP=(AP)_M\oplus (AP)_N=0\oplus A_N
$$ and Lemma \ref{Riesz} (i) it follows that $AP$ is Riesz. As $A_N$ is Riesz, we have that $A_N+I_N$ is Browder \cite[Theorem 3.2]{SANU}, and hence it is left Browder. From
  $$
A+P=(A+P)_M\oplus (A+P)_N=A_M\oplus (A_N+I_N),
$$ according to Lemma \ref{Riesz} (iv) we conclude that $A+P$ is left Browder.

\smallskip

(iv)$\Longrightarrow$(iii): Suppose that there  exists a projector $P\in L(X)$   commuting with $A$
such that $A+P$ is left Browder and $AP$ is Riesz. Let $M=\N(P)$ and $N=\R(P)$. Then $(M,N)\in Red(A)$, $(M,N)\in Red(AP)$ and $(M,N)\in Red(A+P)$.  From Lemma \ref{Riesz} (i) we obtain  that  $A_N=(AP)_N$ is Riesz, while from Lemma \ref{Riesz} (iv) we have  that $A_M=(A+P)_M$ is left Browder.
From \cite[Theorem 5]{ZDH1} it follows that there exist closed subspace $M_1,M_2$ of $M$ such that $\dim M_2<\infty$,  $(M_1,M_2)\in Red(A_M)$, $A_M=A_{M_1}\oplus A_{M_2}$, $A_{M_1}$ is left invertible and $A_{M_2}$ is nilpotent. Then $N_1=M_2\oplus N$ is a closed subspace of $X$, $(M_1,N_1)\in Red(A)$,  $A_{N_1}=A_{M_2}\oplus A_{N}$ is Riesz according to Lemma \ref{Riesz} (i).

\smallskip

(iv)$\Longrightarrow$(v): Let there exist a projector $P\in L(X)$   commuting with $A$
such that $A+P$ is left Browder and $AP$ is Riesz. Then
 $P$ commutes with $A+P$ and $I$, and since $A+P$ and $I$ are  left Browder, from  Lemma  \ref{almost+Fredholm} (i) it follows that $A(I-P)+P=(A+P)(I-P)+ I\cdot P$ is left Browder.

\smallskip

(v)$\Longrightarrow$(vi): Let $P$ be a projector  which commutes with $A$ and
such that  $A(I-P)+P$ is left Browder and $AP$ is Riesz. Set $C=A(I-P)+P$ and $D=AP$. Then $C$ is left Browder, $D$ is  Riesz, $P$ commutes with $C$ and $D$,   and
$$
C(I-P)+DP=(A(I-P)+P)(I-P)+AP=A.$$

(vi)$\Longrightarrow$(iii):
 Suppose that there exist $C,D,P\in L(X)$ such that $C$ is  left Browder, $D$ is Riesz, $P$ is a projector  commuting with $C$ and $D$,  and  $A=C(I-P)+DP$.
Then $P$ is commuting with $C(I-P)$ and $DP$, and hence $(\R(P),\N(P))\in Red(C(I-P))$ and  $(\R(P),\N(P))\in Red(DP)$. Consequently,
\begin{eqnarray*}
  A &=& C(I-P)+DP=((C(I-P))_{\R(P)}\oplus (C(I-P))_{\N(P)})+((DP)_{\R(P)}\oplus (DP)_{\N(P)})\\&=&(0_{\R(P)}\oplus C_{\N(P)})+(D_{\R(P)}\oplus 0_{\N(P)}) =D_{\R(P)}\oplus C_{\N(P)}.
\end{eqnarray*}
 Lemma \ref{Riesz} (i) ensures that $D_{\R(P)}$ is Riesz, while Lemma \ref{Riesz} (iv) ensures that $C_{\N(P)}$ is left Browder. Now as in the proof of the implication (iv)$\Longrightarrow$(iii) we conclude that $A$ is a direct sum of a Riesz operator and a left invertible operator.

\smallskip

 (i)$\Longrightarrow$(vii): Suppose that  $A$ is left generalized Drazin-Riesz invertible. Then there exists
$B\in L(X)$ such that
$
ABA=BA^2,\ B^2A=B, \ A-ABA \ {\rm is\ Riesz}.
$
Set $Q=BA$. Then
$$
Q^2=BABA=B^2A^2=BA=Q,\quad AQ=ABA=BA^2=QA,
$$
 $A(I-Q)=A-ABA$ is Riesz and $Q\in L(X)A$. Therefore, $A$ is left Riesz-quasipolar.

 \smallskip

 (vii)$\Longrightarrow$(i): Let $A$ be left Riesz-quasipolar. Thus there exists  a projector  $Q\in L(X)$ such that $AQ=QA,\ A(I-Q)\ {\rm is\ Riesz},\ Q\in L(X)A$.
Then there exists $T\in L(X)$ such that $Q=TA$. Set $B=QTQ$. Then $ABA=BA^2$, $B^2A=B$ and $A-ABA=A-AQ=A(I-Q)\ {\rm is\ Riesz}$,
%\begin{eqnarray*}
%% \nonumber to remove numbering (before each equation)
%  && ABA=AQTQA=AQTAQ=AQ,   BA^2=QTQA^2=QTA^2Q=QAQ=AQ \Longrightarrow ABA=BA^2,
%  \\&& B^2A=QTQQTQA=QTQTAQ=QTQ=B,  \\&&
%  A-ABA=A-AQ=A(I-Q)\ {\rm is\ Riesz},
%\end{eqnarray*}
and so $A$ is left generalized Drazin-Riesz invertible.

\smallskip

(iii)$\Longrightarrow$(viii): Suppose that there exists $(M,N)\in Red(A)$ such that $ A_M$ is left invertible   and $ A_N$ is Riesz. Then $A_M$ is Saphar, and so $A$ admits a GSRD$(M,N)$.
 Further there exists an $\epsilon>0$ such that $D(0,\epsilon)\cap \sigma_{l}(A_M)=\emptyset$. Since   $A_N$ is Riesz, its left spectrum $\sigma_{l}(A_N)$ is at most countable with $0$ as its only possible limit point.  As  $ \sigma_{l}(A)= \sigma_{l}(A_M)\cup \sigma_{l}(A_N)$ according Lemma \ref{suma left}, we conclude that $0\notin{\rm int}\, \sigma_{l}(A)$.

\smallskip

(viii)$\Longrightarrow$(ix): Let $0\notin {\rm int}\, \sigma_{l} (A)$. This implies that $0\notin\sigma_{l} (A)$ or $0\in\partial\sigma_{l} (A)$. In both cases  $A$ has SVEP at 0
by
the identity theorem for analytic functions.

\smallskip

(ix)$\Longrightarrow$(iii):
Suppose that $A$ admits a GSRD and $A$ has the SVEP at $0$. Then there exists $(M,N)\in Red(A)$ such that $ A_M$ is Saphar and $ A_N$ is Riesz. Further we conclude that  $A_M$ has the SVEP at $0$ and according  \cite[Theorem 2.49]{Ai} it follows  that $A_M$ is left invertible.

\smallskip

 (iii)$\Longrightarrow$(x):  Suppose that
  there exists $(M,N)\in Red(A)$ such that $ A_M$ is left invertible  and $ A_N$ is Riesz. Then   $A$ admits a
GSRD$(M,N)$. As $A_M$ is left invertible,  there exists $\epsilon>0$ such that for every $\lambda\in \CC$ satisfying $|\lambda|<\epsilon$ we have  $A_M-\lambda I_M$ is left invertible. Since  $A_N$ is Riesz, it follows that  $A_N-\lambda I_N$ is   left Browder   for every $\lambda\in \CC$ such that $0<|\lambda|<\epsilon$  \cite[Theorem 3.2]{SANU}.   Lemma  \ref{Riesz} (iv) ensures that
 $A-\lambda I$ is left Browder  for every $\lambda\in\CC$ such that $0<|\lambda|<\epsilon$, and so $0\notin \acc\, \sigma_{\B}^l(A)$.

 \smallskip

The implications (x)$\Longrightarrow$(xi)$\Longrightarrow$(xiii), (x)$\Longrightarrow$(xii)$\Longrightarrow$(xiii), (viii)$\Longrightarrow$(xiv) are obvious.

\smallskip
(xiii)$\Longrightarrow$(iii), (xiv)$\Longrightarrow$(iii):  Suppose that $A$ admits GSRD and $0 \notin \inter\, \sigma_{D}^+(A)$ (resp., $0 \notin \inter\, \sigma_{ p}(A)$). Then  there exists a decomposition $(M,N)\in Red(A)$ such that $ A_M$ is Saphar  and $A_N$ is Riesz. By using Grabiner's punctured neighborhood theorem \cite[Theorem 4.7]{Grabiner}, from the proof of 
of the implication (xii)$\Longrightarrow$(i) (resp., (xiv)$\Longrightarrow$(i)) in \cite[Theorem 3.19]{MZ} we conclude that $A_M$ is left invertible.
\end{proof}

Characterizations of right generalized Drazin-Riesz invertible operators are obtained in a similar way as those of left  generalized Drazin-Riesz invertible operators in Theorem \ref{prva Riesz}.

\begin{theorem} \label{druga Riesz} For $A\in L(X)$ the following conditions are equivalent:

\snoi {\rm (i)}
 $A$ is right generalized Drazin-Riesz invertible;

\snoi {\rm (ii)}  There exists $C\in L(X)$ such that
\begin{equation*}%\label{drugi uslovi right}
    ACA=A^2C,\ AC^2=C=CAC, \ A-ACA \ {\rm is\ Riesz};
\end{equation*}

\snoi {\rm (iii)} There exists $(M,N)\in Red(A)$ such that $ A_M$ is right invertible   and $ A_N$ is Riesz;

\snoi {\rm (iv)}
There exists a projector $P\in L(X)$   commuting with $A$
such that $A+P$ is right Browder and $AP$ is Riesz;

\snoi {\rm (v)} There exists a projector $P\in L(X)$   commuting with $A$
such that $A(I-P)+P$ is right Browder and $AP$ is Riesz;

\snoi {\rm (vi)}  There exist $C,D,P\in L(X)$ such that $C$ is  right Browder, $D$ is Riesz, $P$ is a projector  commuting with $C$ and $D$,  and  $A=C(I-P)+DP$;

\snoi {\rm (vii)}
$A$ is right Riesz-quasipolar;

\snoi {\rm (viii)} $A$ admits a GSRD  and
$0\notin {\rm int}\, \sigma_{r}(A)$;

\snoi {\rm (ix)} $A$ admits a GSRD and $A^\prime$ has the SVEP at $0$;

\snoi {\rm (x)} $A$ admits a GSRD  and
$0\notin {\rm acc}\, \sigma_{\B}^r(A)$;

\snoi {\rm (xi)} $A$ admits a GSRD  and
$0\notin {\rm int}\, \sigma_{\B}^r(A)$;

\snoi {\rm (xii)} $A$ admits a GSRD  and
$0\notin {\rm acc}\, \sigma_{dsc}(A)$;

\snoi {\rm (xiii)} $A$ admits a GSRD  and
$0\notin {\rm int}\, \sigma_{dsc}(A)$;

%\snoi {\rm (xiv)} $A$ admits a GSRD  and
%$0\notin {\rm acc}\, \sigma_{p}(A)$;

\snoi {\rm (xiv)} $A$ admits a GSRD  and
$0\notin {\rm int}\, \sigma_{cp}(A)$.

%?? \snoi {\rm (ix)} There exists a projector $P\in L(X)$
%which commutes with $A$ such that $(I-P)A(I-P)$ is right invertible in $(I-P)L(X)(I-P)$ and $AP$ is Riesz  in $PL(X)P$.

\end{theorem}

\begin{remark}\label{ne mogu} \rm
From   Theorem  \ref{prva Riesz}, Theorem  \ref{druga Riesz} and \cite[Theorems 3.3 and 3.4]{Thome},
   \cite[Theorems 3.19 and 3.20]{MZ}
  it follows that every left (right) generalized   Drazin invertible operator  is left generalized Drazin-Riesz invertible.
  %So the concept of left (right) Drazin-Riesz invertible operators is an extension of the concept of left (right) generalized   Drazin invertible operators.

The condition that $0\notin{\rm int}\, \sigma_{l }(A)$    ($0\notin{\rm int}\, \sigma_{r}(A)$) in the statement (viii) in Theorem \ref{prva Riesz} (Theorem \ref{druga Riesz})  can not be replaced with the stronger condition that $0\notin{\rm acc}\, \sigma_{l}(A)$ ($0\notin{\rm acc}\, \sigma_{r}(A)$).
Also, the condition that $0\notin{\rm int}\, \sigma_{p }(A)$    ($0\notin{\rm int}\, \sigma_{cp}(A)$) in the statement (xiv) in Theorem \ref{prva Riesz} (Theorem \ref{druga Riesz}), can not be replaced with the stronger condition that $0\notin{\rm acc}\, \sigma_{p}(A)$ ($0\notin{\rm acc}\, \sigma_{cp}(A)$).
Namely,  if $A\in L(X)$ is a Riesz operator with infinite spectrum, then  $A$ is left and right generalized Drazin-Riesz  invertible, but $\sigma(A)=\sigma_{l}(A)=\sigma_{r}(A)$, $\sigma(A)\setminus\{0\}\subset\sigma_p(A)$, $\sigma(A)\setminus\{0\}\subset\sigma_{cp}(A)$, and  $0\in{\rm acc}\, \sigma_{l}(A)={\rm acc}\, \sigma_{r}(A)$, as well as $0\in{\rm acc}\, \sigma_{p}(A)$ and $0\in{\rm acc}\, \sigma_{cp}(A)$.

According to  \cite[Theorems 3.19 and 3.20]{MZ} we conclude  that $A$ is  neither left  generalized Drazin invertible nor right  generalized Drazin invertible.
 Hence the class of left (right) generalized Drazin invertible operators is strictly contained in the class of left (right) generalized Drazin-Riesz  invertible operators.
\end{remark}

\begin{corollary} \label{get} Let $A\in L(X)$. Then $A$ is  generalized Drazin-Riesz invertible if and only if $A$ is both left and right generalized Drazin-Riesz invertible.
\end{corollary}
\begin{proof} $(\Longrightarrow)$: It is obvious.

$(\Longleftarrow)$: Let $A$ be both left and right generalized Drazin-Riesz invertible. Then from Theorems \ref{prva Riesz} and \ref{druga Riesz} it follows that $A$ admits a GSRD, and  $0\notin {\rm acc}\, \sigma_{\B}^l(A)$ and $0\notin {\rm acc}\, \sigma_{\B}^r(A)$. Since ${\rm acc}\, \sigma_{\B}^l(A)\cup {\rm acc}\, \sigma_{\B}^r(A)={\rm acc}( \sigma_{\B}^l(A)\cup\sigma_{\B}^r(A))={\rm acc}\, \sigma_{\B}(A)$, we get that $0\notin {\rm acc}\, \sigma_{\B}(A)$. Now from \cite[Theorem 2.3]{ZC}    it follows that $A$ is  generalized Drazin-Riesz invertible.
\end{proof}

\begin{remark}\label{pp} \rm Every left invertible operator  which is not right invertible is left  generalized Drazin-Riesz invertible, but it is not right generalized Drazin-Riesz invertible. Indeed, if $A\in L(X)$ is left invertible and if it is not right invertible, then $A$ is left  generalized Drazin-Riesz invertible, $0\in \sigma_{r}(A)$ and $0\notin \sigma_{l}(A)$. Since $\partial \, \sigma_{r}(A)\subset \sigma_{l}(A) $, we conclude that  $0$ cannot be a boundary point of $\sigma_{r}(A)$. Consequently, $0\in\inter\, \sigma_{r}(A)$, and from Theorem \ref{druga Riesz} it follows that $A$ is not right generalized Drazin-Riesz invertible.

Similarly, every right  invertible operator  which is not  left  invertible is right  generalized Drazin-Riesz invertible, but it is not left generalized Drazin-Riesz invertible.
Therefore, the class of generalized Drazin-Riesz invertible operators is strictly contained in the class of  left (right) generalized Drazin-Riesz invertible operators.
%For example, if $X\in \{  c_0(\NN), c(\NN), \ell_{\infty}(\NN),\ell_p(\NN)\}$, $p\ge 1$, and  $U$ and $V\in L(X)$ are the forward and backward unilateral shifts, then $U$, as a left invertible but not right invertible operator, is left  generalized Drazin-Riesz invertible, but it is not right generalized Drazin-Riesz invertible,  while $V$,
% is right  generalized Drazin-Riesz invertible, but it is not left generalized Drazin-Riesz invertible.
% is right invertible, but it is not left invertible
\end{remark}

\begin{corollary}\label{cup} For  $A\in L(X)$ the following conditions are equivalent:

 \snoi {\rm (i)} $A$ is  generalized Drazin-Riesz invertible;

\snoi {\rm (ii)} $A$ is either  left  generalized Drazin-Riesz invertible or right generalized Drazin-Riesz invertible, and $0\notin {\rm int}\, \sigma(A)$;

\snoi {\rm (iii)} $A$ admits a GSRD and $0\notin {\rm int}\, \sigma(A)$.

\end{corollary}
\begin{proof} (i)$\Longrightarrow$(ii): Let $A$ be generalized Drazin-Riesz invertible. Then $A$ is both left and right generalized Drazin-Riesz invertible and from   \cite[Theorem 2.3]{ZC} %(see the equivalence (ii)$\Longleftrightarrow$(iv))
 it follows that $0\notin {\rm int}\, \sigma(A)$.

(ii)$\Longrightarrow$(iii):  Let $A$ be either  left  generalized Drazin-Riesz invertible or right generalized Drazin-Riesz invertible, and $0\notin {\rm int}\, \sigma(A)$. Then from Theorems \ref{prva Riesz} and \ref{druga Riesz} it follows that $A$ admits a GSRD

(iii)$\Longrightarrow$(i): Let $A$ admit a GSRD and $0\notin {\rm int}\sigma(A)$. Then $A$ admits a GKRD and using      \cite[Theorem 2.3]{ZC}  we conclude that $A$ is  generalized Drazin-Riesz invertible.
% It follows from  \cite[Theorem 2.3]{ZC}.
\end{proof}
\begin{corollary} Let   $A\in L(X)$. If $0\notin {\rm int}\, \sigma(A)$, then $A$ admits a GKRD if and only if $A$ admits a GSRD.
\end{corollary}
\begin{proof} It follows from Corollary \ref{cup} and  \cite[Theorem 2.3]{ZC}.
\end{proof}

\begin{proposition}\label{treca Riesz} Let $A\in L(X)$. Then:

\snoi {\rm (i)} If there exists $(M,N)\in Red(A)$ such that $ A_M$ is left invertible   and $ A_N$ is Riesz, then $(N^\bot, M^\bot)\in Red (A^\prime  )$, ${A^\prime}_{N^\bot}$ is right invertible and ${A^\prime}_{M^\bot}$ is Riesz.

\snoi {\rm (ii)} If there exists $(M,N)\in Red(A)$ such that $ A_M$ is right invertible   and $ A_N$ is Riesz, then $(N^\bot, M^\bot)\in Red (A^\prime  )$, ${A^\prime}_{N^\bot}$ is left invertible and ${A^\prime}_{M^\bot}$ is Riesz.
\end{proposition}
\begin{proof} (i): Suppose that $(M,N)\in Red(A)$, $ A_M$ is left invertible   and $ A_N$ is Riesz.  Let $P$ be the projector  of $X$ such that $\R(P)=M$ and $\N(P)=N$. Then $P^\prime$ is a projector which commutes with $A^\prime$. As $\R(P^\prime)=N^{\bot}$ and $\N(P^{\prime})=M^{\bot}$, we obtain that
  $(N^\bot, M^\bot)\in Red (A^\prime)$. %From the proof of \cite[Theorem 1.43]{Ai} it follows  that ${A^\prime}_{N^\bot}$ is Kato.
  Since $A_M$ is left invertible, we have that $\N(A_M)=\{0\}$ and so
\begin{equation*}\label{po}
    \R(A^{\prime}_{N^{\bot}}) =\R(A^{\prime})\cap N^{\bot}=\N(A)^{\bot}\cap N^{\bot}=(\N(A)+N)^{\bot}
   = (\N(A_M)\oplus N)^{\bot}=N^{\bot},
\end{equation*}
i.e.
   $A^{\prime}_{N^{\bot}}$ is surjective. Further we have that
  \begin{eqnarray}
    \N(A^{\prime}_{N^{\bot}}) =\N(A^{\prime})\cap N^{\bot}=\R(A)^{\bot}\cap N^{\bot}=(\R(A)+N)^{\bot} = (\R(A_M)\oplus N)^{\bot}\label{po1}
  \end{eqnarray}
 Since $A_M$ is left invertible,  it follows that  $\R(A_M)$ is complemented in $M$. According to Lemma \ref{glavna lema} we conclude that $\R(A_M)\oplus N$ is complemented in $X$ and by using Lemma \ref{bot}  we get that  $(\R(A_M)\oplus N)^{\bot}$ is complemented in $X^{\prime}$. As $N^{\bot}$ is a closed subspace of $X^\prime$ which contains  $(\R(A_M)\oplus N)^{\bot}$, applying Lemma \ref{complemented in} we conclude that $(\R(A_M)\oplus N)^{\bot}$ is  complemented in $N^{\bot}$. According to  \eqref{po1} we have that  $\N(A^{\prime}_{N^{\bot}})$ is complemented in $N^{\bot}$, and hence  ${A^\prime}_{N^\bot}$ is right invertible.

 Let $Q=I-P$. Then $Q$ is a projector, $\R(Q)=N$, $\N(Q)=M$,
$(M,N)\in Red (AQ)$,  $AQ=0_M\oplus A_N$, and so  Lemma \ref{Riesz} (i)  ensures that  $AQ$ is Riesz. Also we have that $(N^\bot, M^\bot)\in Red (A^\prime Q^\prime )$.
From \cite[Corollary 3.10]{SANU} it follows that $A^\prime Q^\prime=Q^\prime A^\prime$ is Riesz. %Since $M^\bot$ is a closed subspace of $X^\prime$, from \cite[Lemma 3.5.1]{CPY}
 As   $\R(Q^\prime)=\N(Q)^\bot=M^\bot$, using Lemma \ref{Riesz} (i) we obtain that ${A^\prime Q^\prime}_{M^\bot}=A^\prime _{M^\bot}$ is Riesz.

%Let $Q$ be the projection  of $X$ onto $N$ along $M$.
%Then $(M,N)\in Red (AQ)$,  $AQ=0_M\oplus A_N$, and so  Lemma \ref{Riesz} ensures that  $AQ$ is Riesz. Also we have that $(N^\bot, M^\bot)\in Red (A^\prime Q^\prime )$.
%From \cite[Corollary 3.10]{SANU} it follows that $A^\prime Q^\prime=Q^\prime A^\prime$ is Riesz.
% As   $\R(Q^\prime)=\N(Q)^\bot=M^\bot$ and $\N(Q^\prime)=\R(Q)^\bot=N^\bot$,  we conclude that $A^\prime Q^\prime=(A^\prime Q^\prime)_{N^\bot}\oplus (A^\prime Q^\prime)_{M^\bot}=0_{N^\bot}\oplus A^\prime _{M^\bot}$. Using again   Lemma \ref{Riesz} we obtain that ${A^\prime}_{M^\bot}$ is Riesz. Consequently, $A^\prime$ admits a
%$GSRD(N^{\bot}, M^{\bot})$.

(ii):  Suppose that $(M,N)\in Red(A)$, $ A_M$ is right invertible   and $ A_N$ is Riesz. Then $\N(A_M)$ is a complemented subspace of $M$ and $\R(A_M)=M$.   As in the proof of (i) we conclude that $A^\prime _{M^\bot}$ is Riesz,
 %\begin{eqnarray*}
%  % \nonumber to remove numbering (before each equation)
%    \N(A^{\prime}_{N^{\bot}})\!  =\! \N(A^{\prime})\cap N^{\bot}\! =\! \R(A)^{\bot}\cap N^{\bot}\! =\! (\R(A)+N)^{\bot}\!  = \! (\R(A_M)\oplus N)^{\bot}\! =\! (M\oplus N)^{\bot}\! =\! \{0\},
%  \end{eqnarray*}
  \begin{eqnarray*}
    \N(A^{\prime}_{N^{\bot}})=(\R(A_M)\oplus N)^{\bot} = (M\oplus N)^{\bot}= \{0\},
  \end{eqnarray*}
  that is $A^{\prime}_{N^{\bot}}$ is injective, and
  \begin{equation*}\label{pore}
    \R(A^{\prime}_{N^{\bot}})
   = (\N(A_M)\oplus N)^{\bot}.
\end{equation*}
% \begin{equation}\label{pore}
%    R(A^{\prime}_{N^{\bot}}) =\R(A^{\prime})\cap N^{\bot}=\N(A)^{\bot}\cap N^{\bot}=(\N(A)+N)^{\bot}
%   = (\N(A_M)\oplus N)^{\bot}.
%\end{equation}
 As $\N(A_M)$ is complemented in $M$, from Lemma \ref{glavna lema} it follows that $\N(A_M)\oplus N$ is  complemented in $X$. Using  Lemma \ref{bot} we get that $(\N(A_M)\oplus N)^{\bot}$ is  complemented in $X^{\prime}$. As $N^{\bot}$ is a closed subspace of $X^\prime$ which contains $(\N(A_M)\oplus N)^{\bot}$, applying Lemma \ref{complemented in} we conclude that $(\N(A_M)\oplus N)^{\bot}$ is  complemented in $N^{\bot}$. %Now according to  \eqref{pore}
As $\R(A^{\prime}_{N^{\bot}})$  is complemented in $N^{\bot}$ and $A^{\prime}_{N^{\bot}}$ is injective, we get that ${A^\prime}_{N^\bot}$ is left invertible.
\end{proof}

\begin{corollary} Let $A\in L(X)$. Then:

\snoi {\rm (i)} If $A$ is left generalized Drazin-Riesz invertible,  then $A^\prime$ is  right generalized Drazin-Riesz invertible.

\snoi {\rm (ii)} If  $A$ is right  generalized Drazin-Riesz invertible, then $A^\prime$ is left   generalized Drazin-Riesz invertible.
\end{corollary}
\begin{proof} It follows from Proposition \ref{treca Riesz} and  the equivalence (i)$\Longleftrightarrow$(iii) in Theorem \ref{prva Riesz} and Theorem \ref{druga Riesz}.
\end{proof}

\begin{proposition}\label{GSD} Let $A\in L(X)$.  If $A$ admits a $GSRD(M, N)$, then $A^\prime$  admits a \break
$GSRD(N^{\bot}, M^{\bot})$.
\end{proposition}
\begin{proof} Let $A$ admit a $GSRD(M, N)$. Then  $(M,N) \in Red(A)$,
$A_M$ is Saphar and $A_N$ is Riesz. As in the proof of Proposition \ref{treca Riesz} we get that $(N^\bot, M^\bot)\in Red (A^\prime)$, $A^\prime _{M^\bot}$ is Riesz and   the subspaces  $\N(A^{\prime}_{N^{\bot}})=(\R(A_M)\oplus N)^{\bot}$ and   $\R(A^{\prime}_{N^{\bot}}) =(\N(A_M)\oplus N)^{\bot}$,   are complemented  in $N^{\bot}$.
  From the proof of \cite[Theorem 1.43]{Ai} it follows  that ${A^\prime}_{N^\bot}$ is Kato. Therefore, ${A^\prime}_{N^\bot}$ is Saphar, and so $A^\prime$  admits a
$GSRD(N^{\bot}, M^{\bot})$.
\end{proof}

\section{Left and right generalized  Drazin-meromorphic invertible operators}

In this section we introduce and explore the concept of  generalized Saphar-meromorphic
decomposition and the concept of left and right generalized  Drazin-meromorphic invertible operators. Also we show that  on an arbitrary Banach space  the class of left (resp., right)  generalized Drazin-Riesz invertible operators does not coincide with the class of upper (resp., lower) generalized Drazin-Riesz invertible operators, and  the class of left (resp., right)   generalized Drazin-meromorphic invertible operators does not coincide with the class of upper (resp., lower)  generalized Drazin-meromorphic  invertible operators (see definitions on page 15).

\begin{definition} \rm Let $A\in L(X)$. If  there exists a pair $(M,N)
\in Red(A)$ such that $A_M$ is Saphar and $A_N$ is meromorphic, then we say that $A$ admits  a {\em generalized Saphar-meromorphic
decomposition}, or shortly $A$ admits a GS($\M$)D$(M,N)$ (or   $A$ admits a GS($\M$)D).
% An operator $A \in L(X)$ is said to  admit a {\em generalized Saphar-Riesz
%decomposition},  if there exists a pair $(M,N)
%\in Red(A)$ such that $A_M$ is Kato and $A_N$ is Riesz. In that case we shall  say that $A$ admits a GSRD$(M,N)$, or shortly a GSRD.
\end{definition}

\begin{theorem} Let $A\in L(X)$. The following conditions are equivalent:

\snoi {\rm (i)}  $A$ admits a GS($\M$)D;

\snoi {\rm (ii)} There exists a projector $P\in L(X)$   commuting with $A$
such that $A+P$ is of Saphar type  and $AP$ is meromorphic.
\end{theorem}
\begin{proof}
 It follows from Lemma \ref{Riesz} (ii) analogously to the proof of Theorem \ref{prva S}.
\end{proof}

\begin{definition}\label{left-M-Drazin} \rm An operator
$A\in L(X)$ is  {\it left  generalized Drazin-meromorphic invertible} if there is $B\in L(X)$ such that
\begin{equation*}%\label{prvi uslovi}
ABA=BA^2,\ B^2A=B, \ A-ABA \ {\rm is\ meromorphic}.
\end{equation*}

An operator
$A\in L(X)$ is  {\it right  generalized Drazin-meromorphic invertible} if there is $B\in L(X)$ such that
\begin{equation*}%\label{prvi uslovi r}
ABA=A^2B,\ AB^2=B, \ A-ABA \ {\rm is\ meromorphic}.
\end{equation*}
\end{definition}

It is clear that every   left (right)  generalized Drazin-Riesz invertible operators is  left (right)  generalized Drazin-meromorphic invertible. %$\sigma_{\R}(0)\subset\sigma(0) =\{0\}$
If $A$ is meromorphic, then it is left (right) generalized Drazin-meromorphic invertible and  $0$ is a left (right) generalized Drazin-meromorphic inverse of $A$.

\begin{definition} \rm An operator  $A\in L(X)$ is {\it left meromorphic-quasipolar} if there exists a projector   $Q\in L(X)$ satisfying
\begin{equation*}
 AQ=QA,\ A(I-Q)\ {\rm is\ meromorphic},\ Q\in L(X)A.
\end{equation*}
An operator $A\in L(X)$ is {\it right meromorphic-quasipolar} if there exists a projector $Q\in L(X)$ satisfying
\begin{equation*}
 AQ=QA,\ A(I-Q)\ {\rm is\ meromorphic },\ Q\in AL(X).
\end{equation*}
\end{definition}

In the following theorem we give some characterizations of left generalized Drazin-meromorphic invertible operators.

\begin{theorem} \label{prva Meromorphic} For $A\in L(X)$ the following conditions are equivalent:

\snoi {\rm (i)}
 $A$ is left generalized Drazin-meromorphic invertible;

\snoi {\rm (ii)}  There exists $C\in L(X)$ such that
\begin{equation*}%\label{drugi uslovi m}
    ACA=CA^2,\ C^2A=C=CAC, \ A-ACA \ {\rm is\ meromorphic};
\end{equation*}

\snoi {\rm (iii)} There exists $(M,N)\in Red(A)$ such that $ A_M$ is left invertible   and $ A_N$ is meromorphic;

\snoi {\rm (iv)}
There exists a projector $P\in L(X)$   commuting with $A$
such that $A+P$ is left Drazin invertible and $AP$ is meromorphic;

\snoi {\rm (v)} There exists a projector $P\in L(X)$   commuting with $A$
such that $A(I-P)+P$ is left Drazin invertible and $AP$ is meromorphic;

\snoi {\rm (vi)}  There exist $C,D,P\in L(X)$ such that $C$ is  left Drazin invertible, $D$ is meromorphic, $P$ is a projector  commuting with $C$ and $D$,  and  $A=C(I-P)+DP$;

\snoi {\rm (vii)}
$A$ is left meromorphic-quasipolar;

\snoi {\rm (viii)} $A$ admits a GS($\M$)D  and
$0\notin {\rm int}\, \sigma_{l}(A)$;

\snoi {\rm (ix)} $A$ admits a GS($\M$)D and $A$ has the SVEP at $0$;

\snoi {\rm (x)} $A$ admits a GS($\M$)D  and
$0\notin {\rm acc}\, \sigma_{D}^l(A)$;

\snoi {\rm (xi)} $A$ admits a GS($\M$)D  and
$0\notin {\rm int}\, \sigma_{D}^l(A)$;

\snoi {\rm (xii)} $A$ admits a GS($\M$)D  and
$0\notin {\rm acc}\, \sigma_{D}^+(A)$;

\snoi {\rm (xiii)} $A$ admits a GS($\M$)D  and
$0\notin {\rm int}\, \sigma_{D}^+(A)$;

%\snoi {\rm (xiv)} $A$ admits a GS($\M$)D  and
%$0\notin {\rm acc}\, \sigma_{p}(A)$;

\snoi {\rm (xiv)} $A$ admits a GS($\M$)D  and
$0\notin {\rm int}\, \sigma_{p}(A)$.

%?? \snoi {\rm (ix)} There exists a projector $P\in L(X)$
%which commutes with $A$ such that $(I-P)A(I-P)$ is left invertible in $(I-P)L(X)(I-P)$ and $AP$ is Meromorphic  in $PL(X)P$.

\end{theorem}

\begin{proof} (iii)$\Longrightarrow$(iv): Let   there exist $(M,N)\in Red(A)$ such that $ A_M$ is left invertible   and $ A_N$ is meromorphic.  Let $P$ be the projector  of $X$ such that $\R(P)=N$ and $\N(P)=M$. Then  $P$ commutes with $A$, $AP$ and $A+P$ and hence $(M,N)\in Red (AP)$ and $(M,N)\in Red (A+P)$. From
$$
AP=(AP)_M\oplus (AP)_N=0\oplus A_N
$$ according to Lemma \ref{Riesz} (ii) we get that $AP$ is meromorphic. Since $A_N$ is meromorphic, it follows that  $A_N+I_N$ is Drazin invertible and hence  it is left Drazin invertible. From
  $$
A+P=(A+P)_M\oplus (A+P)_N=A_M\oplus (A_N+I_N),
$$ using Lemma \ref{left decom} we conclude that $A+P$ is left Drazin invertible.

(iv)$\Longrightarrow$(iii): Suppose that there  exists a projector $P\in L(X)$   commuting with $A$
such that $A+P$ is left Drazin invertible  and $AP$ is Riesz. Let $M=\N(P)$ and $N=\R(P)$. Then $(M,N)\in Red(A)$, $(M,N)\in Red(AP)$ and $(M,N)\in Red(A+P)$. From Lemma \ref{Riesz} (ii) we obtain that $A_N=(AP)_N$ is meromorphic, while from Lemma \ref{left decom} we get that $A_M=(A+P)_M$ is left Drazin invertible.
From \cite[Theorem 3.4]{Ghorbel Mnif}  it follows that there exist closed subspace $M_1,M_2$ of $M$ such that   $(M_1,M_2)\in Red(A_M)$, $A_M=A_{M_1}\oplus A_{M_2}$, $A_{M_1}$ is left invertible and $A_{M_2}$ is nilpotent. From Lemma \ref{zatvorenost}  it follows that $N_1=M_2\oplus N$ is a closed subspace of $X$, and so $(M_1,N_1)\in Red(A)$. Lemma \ref{Riesz} (ii) ensures that   $A_{N_1}=A_{M_2}\oplus A_{N}$ is meromorphic.

% The rest of implications is proved by using Lemmas \ref{left Drazin}, \ref{left decom} and \ref{Riesz} in a similar way as the appropriate implications in the proof of Theorem \ref{prva Riesz}.
%
 The proof of the rest of the implication can be obtained by using Lemmas \ref{left Drazin}, \ref{left decom} and \ref{Riesz} (ii)  adapting the appropriate parts of the proof of Theorem \ref{prva Riesz}.
\end{proof}

The following theorem provides characterizations of right generalized Drazin-meromorphic
 invertible operators and  is obtained in a similar way.

\begin{theorem} \label{druga Meromorphic} For $A\in L(X)$ the following conditions are equivalent:

\snoi {\rm (i)}
 $A$ is right generalized Drazin-meromorphic invertible;

\snoi {\rm (ii)}  There exists $C\in L(X)$ such that
\begin{equation*}%\label{drugi uslovi right m}
    ACA=A^2C,\ AC^2=C=CAC, \ A-ACA \ {\rm is\ meromorphic};
\end{equation*}

\snoi {\rm (iii)} There exists $(M,N)\in Red(A)$ such that $ A_M$ is right invertible   and $ A_N$ is meromorphic;

\snoi {\rm (iv)}
There exists a projector $P\in L(X)$   commuting with $A$
such that $A+P$ is right Drazin invertible and $AP$ is meromorphic;

\snoi {\rm (v)} There exists a projector $P\in L(X)$   commuting with $A$
such that $A(I-P)+P$ is right Drazin invertible and $AP$ is meromorphic;

\snoi {\rm (vi)}  There exist $C,D,P\in L(X)$ such that $C$ is  right Drazin invertible, $D$ is meromorphic, $P$ is a projector  commuting with $C$ and $D$,  and  $A=C(I-P)+DP$;

\snoi {\rm (vii)}
$A$ is right meromorphic-quasipolar;

\snoi {\rm (viii)} $A$ admits a GS($\M$)D  and
$0\notin {\rm int}\, \sigma_{r}(A)$;

\snoi {\rm (ix)} $A$ admits a GS($\M$)D and $A^\prime$ has the SVEP at $0$;

\snoi {\rm (x)} $A$ admits a GS($\M$)D  and
$0\notin {\rm acc}\, \sigma_{D}^r(A)$;

\snoi {\rm (xi)} $A$ admits a GS($\M$)D  and
$0\notin {\rm int}\, \sigma_{D}^r(A)$;

\snoi {\rm (xii)} $A$ admits a GS($\M$)D  and
$0\notin {\rm acc}\, \sigma_{dsc}(A)$;

\snoi {\rm (xiii)} $A$ admits a GS($\M$)D  and
$0\notin {\rm int}\, \sigma_{dsc}(A)$;

%\snoi {\rm (xiv)} $A$ admits a GS($\M$)D  and
%$0\notin {\rm acc}\, \sigma_{p}(A)$;

\snoi {\rm (xiv)} $A$ admits a GS($\M$)D  and
$0\notin {\rm int}\, \sigma_{cp}(A)$.

\end{theorem}

From Remark \ref{ne mogu} it follows that the condition that $0\notin{\rm int}\, \sigma_{l }(A)$    ($0\notin{\rm int}\, \sigma_{r}(A)$) in the statement (viii) in Theorem \ref{prva Meromorphic} (Theorem \ref{druga Meromorphic})  can not be replaced with the stronger condition that $0\notin{\rm acc}\, \sigma_{l}(A)$ ($0\notin{\rm acc}\, \sigma_{r}(A)$).
Also, the condition that $0\notin{\rm int}\, \sigma_{p }(A)$    ($0\notin{\rm int}\, \sigma_{cp}(A)$) in the statement (xiv) in Theorem \ref{prva Meromorphic} (Theorem \ref{druga Meromorphic}), can not be replaced with the stronger condition that $0\notin{\rm acc}\, \sigma_{p}(A)$ ($0\notin{\rm acc}\, \sigma_{cp}(A)$).

The following example shows that the class of left (right) generalized Drazin-Riesz invertible operators is strictly contained in the class of left (right) generalized Drazin-meromorphic invertible operators.

\begin{example}\rm  Let $X$ be an infinite dimensional Banach  space and let
$\tilde X=\oplus_{n=1}^{\infty} X_i$  where $X_i=X$, $i\in\NN$. Let $A=\oplus_{n=1}^{\infty}(\frac 1n I)$ where $I$ is the  identity on $X$. In \cite[Example 3.9]{ZC} it is shown that
$A\in L(\tilde X)$ is a meromorphic operator which does not admit a GKRD. Hence $A$ is both  left and right generalized Drazin-meromorphic invertible, and $A$  does not admit a GSRD. From Theorems \ref{prva Riesz} and  \ref{druga Riesz}  it follows that $A$ is  neither left  generalized Drazin-Riesz invertible nor right  generalized Drazin-Riesz invertible.
\end{example}

\begin{theorem} Let $A\in L(X)$ be left (resp., right) generalized Drazin-meromorphic invertible. Then there exists an operator  $B\in L(X)$ commuting with $A$ such that $B$ is  Drazin invertible, $A+B$ is left (resp., right) invertible and  $AB$ is meromorphic.
\end{theorem}
\begin{proof} Suppose that    $A$ is left  generalized Drazin-meromorphic invertible. According to Theorem \ref{prva Meromorphic} 
 there exists a pair $(M,N)\in Red(A)$ such that $A_M$ is left invertible and $A_N$ is meromorphic. Then  $I_N-A_N$ is Drazin invertible and hence there exists a pair $(N_1,N_2)\in Red(I_N-A_N)$ such that $I_{N_1}-A_{N_1}$ is nilpotent and $I_{N_2}-A_{N_2}$ is invertible \cite[Theorem 4]{King}, \cite[Lemma 3.4.2]{CPY}.  Let $B=0_M\oplus (I_N-A_N)$. According Lemma \ref{zatvorenost} we have that $M\oplus N_1$ is closed, and so $(M\oplus N_1,N_2)\in Red(B)$. As  $B_{M\oplus N_1}=0_M\oplus (I_{N_1}-A_{N_1})$ is nilpotent and $B_{N_2}=I_{N_2}-A_{N_2}$ is invertible, it follows that  $B$ is Drazin invertible. Since $A+B=A_M\oplus I_N$, according  Lemma \ref{suma left} we get that $A+B$ is left invertible. Further we have that 
\begin{eqnarray}
     AB&=&(A_M\oplus A_N)(0_M\oplus (I_N-A_N))= 0_M\oplus A_N(I_N-A_N)\label{klj}\\&=&0_M\oplus (I_N-A_N)A_N=(0_M\oplus (I_N-A_N))(A_M\oplus A_N)=BA.\nonumber
\end{eqnarray}
 Since $A_N$ is meromorphic,  $\sigma_D(A_N)\subset\{0\}$. Using the  spectral mapping theorem  \cite[Corollary 2.4]{BS}  for $f(\lambda)=\lambda-\lambda^2$, $\lambda\in\CC$, we conclude  that $\sigma_D( A_N(I_N-A_N))=\sigma_D(f(A_N))=f(\sigma_D(A_N))\subset\{0\}$. Consequently, $A_N(I_N-A_N)$ is meromorphic, and from \eqref{klj} according Lemma \ref{Riesz} (ii) we obtain that $AB$ is meromorphic.
 
 The rest of assertion can be proved similarly.
\end{proof}

%The following theorem which refers to  left (right) generalized Drazin-Riesz invertible operators can be proved similarly.
%\begin{theorem} Let $A\in L(X)$ be left (resp., right) generalized Drazin-Riesz invertible. Then there exists an operator  $B\in L(X)$ commuting with $A$ such that $B$ is  Browder, $A+B$ is left (resp., right) invertible and  $AB$ is Riesz.
%\end{theorem}

\begin{corollary} Let $A\in L(X)$. Then $A$ is  generalized Drazin-meromorphic invertible if and only if $A$ is both left and right generalized Drazin-meromorphic invertible.
\end{corollary}
\begin{proof} $(\Longrightarrow)$: It is obvious.

$(\Longleftarrow)$: Let $A$ be both left and right generalized Drazin-meromorphic invertible. Then from Theorems \ref{prva Meromorphic} and \ref{druga Meromorphic} it follows that $A$ admits a GS($\M$)D, and  $0\notin {\rm acc}\, \sigma_{D}^l(A)$ and $0\notin {\rm acc}\, \sigma_{D}^r(A)$. Hence
$0\notin {\rm acc}( \sigma_{D}^l(A)\cup\sigma_{D}^r(A))={\rm acc}\, \sigma_{D}(A)$, and according  to \cite[Theorem 5]{ZD}    we get  that $A$ is  generalized Drazin-meromorphic invertible.
\end{proof}

%The following example shows that the class of left (right) generalized Drazin-Riesz invertible operators is strictly contained in the class of left (right) generalized Drazin-meromorphic invertible operators.
%
%\begin{example}\rm  Let $X$ be an infinite dimensional Banach  space and let
%$\tilde X=\oplus_{n=1}^{\infty} X_i$  where $X_i=X$, $i\in\NN$. Let $A=\oplus_{n=1}^{\infty}(\frac 1n I)$ where $I$ is identity on $X$. In \cite[Example 3.9]{ZC} it is shown that
%$A\in L(\tilde X)$ is a meromorphic operator which does not admit a GKRD. Hence $A$ is both  left and right generalized Drazin-meromorphic invertible, and $A$  does not admit a GSRD. From Theorems \ref{prva Riesz} and  \ref{druga Riesz}  it follows that $A$ is  neither left  generalized Drazin-Riesz invertible nor right  generalized Drazin-Riesz invertible.
%\end{example}

Similarly as in Remark \ref{pp} we conclude that every left (right) invertible operator  which is not right (left) invertible is left (right) generalized Drazin-meromorphic  invertible, but it is not right (left) generalized Drazin-meromorphic invertible. %Further,

%This shows that the class of  generalized Drazin-meromorphic invertible operators is strictly contained in the class of
%left (right) generalized Drazin-meromorphic invertible operators.

\begin{corollary} \label{cupu} For $A\in L(X)$ the following conditions are equivalent:

 \snoi {\rm (i)} $A$ is  generalized Drazin-meromorphic invertible;

\snoi {\rm (ii)} $A$ is either  left  generalized Drazin-meromorphic invertible or right generalized Drazin-meromorphic invertible, and $0\notin\,  {\rm int}\sigma(A)$;

\snoi {\rm (iii)} $A$ admits a GS($\M$)D and $0\notin \, {\rm int}\sigma(A)$.

% Let $A\in L(X)$. Then $A$ is  generalized Drazin-meromorphic  invertible if and only if $A$ is either  left  generalized Drazin-meromorphic  invertible or right generalized Drazin-meromorphic  invertible, and $0\notin {\rm int}\sigma(A)$.
\end{corollary}
\begin{proof} It follows  from Theorems \ref{prva Meromorphic} and \ref{druga Meromorphic}, and \cite[Theorem 5]{ZD}.
\end{proof}
\begin{corollary} Let $A\in L(X)$. If  $0\notin \, {\rm int}\sigma(A)$, then $A$ admits a GK($\M$)D if and only if $A$ admits a GS($\M$)D.
\end{corollary}
\begin{proof} It follows  from Corollary \ref{cupu} and \cite[Theorem 5]{ZD}.
\end{proof}

\smallskip

Since the adjoint of a meromorphic operator is also meromorphic the following assertions are proved in a similar way as appropriate assertions in Section \ref{left Riesz section}.
\begin{proposition}\label{treca Meromorphic} Let $A\in L(X)$. Then:

\snoi {\rm (i)} If there exists $(M,N)\in Red(A)$ such that $ A_M$ is left invertible   and $ A_N$ is meromorphic, then $(N^\bot, M^\bot)\in Red (A^\prime  )$, ${A^\prime}_{N^\bot}$ is right invertible and ${A^\prime}_{M^\bot}$ is meromorphic.

\snoi {\rm (ii)} If there exists $(M,N)\in Red(A)$ such that $ A_M$ is right invertible   and $ A_N$ is meromorphic, then $(N^\bot, M^\bot)\in Red (A^\prime  )$, ${A^\prime}_{N^\bot}$ is left invertible and ${A^\prime}_{M^\bot}$ is meromorphic.
\end{proposition}

\begin{corollary} Let $A\in L(X)$. Then:

\snoi {\rm (i)} If $A$ is left generalized Drazin-meromorphic invertible,  then $A^\prime$ is  right generalized Drazin-meromorphic invertible.

\snoi {\rm (ii)} If  $A$ is right  generalized Drazin-meromorphic invertible, then $A^\prime$ is left   generalized Drazin-meromorphic invertible.
\end{corollary}

\begin{proposition}\label{GSD+} Let $A\in L(X)$.  If $A$ admits a GS($\M$)D$(M, N)$, then $A^\prime$  admits a \break
GS($\M$)D$(N^{\bot}, M^{\bot})$.
\end{proposition}

 If for an operator $A\in L(X)$ there exists a pair $(M,N)\in Red(A)$ such that $A_M$ is bounded bellow (resp. surjective) and $A_N$ is Riesz, we say that $A$ is  {\em upper  generalized Drazin-Riesz invertible} (resp., {\it lower generalized Drazin-Riesz invertible}). Such kind of  operators was considered in  \cite[Theorems 2.4 and 2.5]{ZC}.  Furthermore, if there exists a pair $(M,N)\in Red(A)$ such that $A_M$ is bounded bellow (resp. surjective) and $A_N$ is meromorphic, then we  say that $A$ is  {\em upper  generalized Drazin-meromorphic invertible} (resp., {\it lower generalized Drazin-meromorphic  invertible}) (see \cite[Theorems 6 and 7]{ZD}).
On a Hilbert space the class of upper (resp., lower) generalized Drazin-Riesz invertible operators coincides with the class of left (resp., right)  generalized Drazin-Riesz invertible operators, and  the class of upper (resp., lower) generalized Drazin-meromorphic invertible operators coincides with the class of left (resp., right)  generalized Drazin-meromorphic  invertible operators.

In order to prove that it does not hold on an arbitrary Banach space we need the  following theorem that is proved in \cite{GZ}.

\begin{theorem}\label{nova}\cite[Theorem 3.1]{GZ} Let $A\in L(X)$ is an upper (resp. lower) Drazin invertible operator. Then the following conditions are equivalent:

%\snoi {\rm (i)}
%  $ A  $ is  left (resp. right) Drazin invertible.

  \snoi {\rm (i)}
 $ A$ is of Saphar type;

\snoi {\rm (ii)}  $ A $  admits a  GSRD;

\snoi {\rm (iii)}
$ A$  admits a  GS$\M$D.

\end{theorem}
\begin{proof} (i)$\Longrightarrow$(ii)$\Longrightarrow$(iii) It is clear.

(iii)$\Longrightarrow$(ii) Suppose that $ T $  admits a  generalized Saphar-meromorphic decomposition. Then there exists a pair $(M,N)\in Red(T)$ such that $T_M$ is Saphar and $T_N$ is meromorphic. It follows that $T_N-\lambda I_N$ is Drazin invertible for every $\lambda\in\CC$, $\lambda\ne 0$, and hence $T_N-\lambda I_N$ is upper Drazin invertible for every $\lambda\ne 0$. Since $T$ is upper Drazin invertible, we have  that $T_N$ is upper Drazin invertible. Therefore, $\{\lambda\in\CC:T_N-\lambda I_N\ {\rm is\ not\ upper\ Drazin\ invertible}\}=\emptyset$, which according to \cite[Corollary 3.15]{Q. Jiang} implies that the Drazin spectrum of $T_N$ is empty. Thus $T_N$ is Drazin invertible, and hence there  exists a pair $(N_1,N_2)\in Red(T_N)$ such that $T_{N_1}$ is invertible  and $T_{N_2}$ is nilpotent. Applying Lemma \ref{Riesz} (iii) we obtain that $T_M\oplus T_{N_1}$ is Saphar. Consequently, $T$ is of Saphar type.

The assertion for the case when $A$ is a lower  Drazin invertible operator  can be proved similarly.
\end{proof}

% We prove that it does not hold on an arbitrary Banach space.

%Namely,
\begin{remark}\label{nakon} \rm  L. Burlando proved that  on a  Banach space $X$ there  exists  an operator $A\in L(X)$ which  is bounded below (resp., surjective) but its  range (resp., null-space) is not complemented \cite[Example 1.1]{Burlando}. Clearly this operator is upper generalized Drazin-Riesz invertible and upper generalized Drazin-meromorphic  invertible  (resp., lower generalized Drazin-Riesz invertible and lower generalized Drazin-meromorphic  invertible), but it is  not of Saphar type (see \cite[p. 170]{ZS}). Hence  from Theorem \ref{nova} it follows that $A$ admits neither  a generalized Saphar-Riesz decomposition nor a generalized Saphar-meromorphic decomposition. Now according to Theorems \ref{prva Riesz} and  \ref{prva Meromorphic} (resp., Theorems \ref{druga Riesz} and  \ref{druga Meromorphic})  it follows that $A$ is neither left generalized Drazin-Riesz invertible nor left generalized Drazin-meromorphic  invertible (resp., neither right generalized Drazin-Riesz invertible nor right generalized Drazin-meromorphic  invertible).
\end{remark}

%This example shows that on an arbitrary Banach space  the class of upper (resp., lower) generalized Drazin-Riesz invertible operators does not coincide with the class of left (resp., right)  generalized Drazin-Riesz invertible operators, and  the class of upper (resp., lower) generalized Drazin-meromorphic invertible operators does not coincide with the class of left (resp., right)  generalized Drazin-meromorphic  invertible operators.

\section{Spectra}

In this section we introduce and study spectra corresponding to operators and decompositions introduced in the previous two sections.

For $A\in L(X)$, set

\begin{eqnarray*}
 \sigma_{ g SR}(A)  &=& \{\lambda \in \mathbb{C}:A-\lambda  I \mbox{ does\ not\ admit\ a}\ GSRD\} ,\\
  \sigma_{ g S\M}(A)  &=&  \{\lambda \in \mathbb{C}:A-\lambda  I \mbox{ does\ not\ admit\ a}\ GS(\M) D\} .
\end{eqnarray*}

The  left generalized   Drazin spectrum, the  right generalized   Drazin spectrum,  the  left generalized   Drazin-Riesz   spectrum, the  right generalized   Drazin-Riesz spectrum, the left generalized   Drazin-meromorphic   spectrum and  the  right generalized   Drazin-meromorphic spectrum
  of $A\in L(X)$ are denoted by $\sigma_{gD}^l(A)$, $\sigma_{gD}^r(A)$, $\sigma_{gDR}^l(A)$, $\sigma_{gDR}^r(A)$, $\sigma_{gD\M}^l(A)$ and $\sigma_{gD\M}^r(A)$,
 respectively.

\begin{theorem}\label{closed}
Let $A\in L(X)$. Then:

\snoi {\rm (i)} If  $A$  admits a GSRD, then there exists $\epsilon>0$ such that $A-\lambda I$ is essentially  Saphar  for each $\lambda$ such that  $0<|\lambda|<\epsilon$;

\snoi {\rm (ii)} If  $A$  admits a GS$(\M)$D, then there exists $\epsilon>0$ such that $A-\lambda I$ is of   Saphar type  for each $\lambda$ such that  $0<|\lambda|<\epsilon$;
\end{theorem}
\begin{proof} There exists  $(M,N)\in Red(A)$, such that $A=A_M\oplus A_N$, $A_M$ is Saphar and $A_N$ is Riesz.
If $M=\{0\}$, then $A$ is Riesz. According to \cite[Theorem 3.2]{SANU} we have that  $A-\lambda I$ is  Fredholm   for all $\lambda\ne 0$. It follows  that $\R(A-\lambda I)$ and $\N(A-\lambda I)$ are complemented subspaces of $X$. According  to \cite[Theorem 16.12]{Mu} we  have that $A-\lambda I$ is  essentially Kato operator for all $\lambda\ne 0$. Hence  $A-\lambda I$ is an essentially Saphar operator for all $\lambda\ne 0$.

Let $M\ne \{0\}$. According  to \cite[Corollary 12.4 and  Lemma 13.6]{Mu} we have  that there exists an $\epsilon>0$ such that   for $|\lambda|<\epsilon$, $A_M-\lambda I_M$ is Saphar. As $A_N$ is Riesz,  $A_N-\lambda I_N$ is  Fredholm, and hence it is essentially Saphar  for all $\lambda\ne 0$. From Lemma \ref{Riesz} (iii) and Lemma \ref{zatvorenost} we obtain that $A-\lambda I$ is essentially  Saphar for each $\lambda$ such that  $0<|\lambda|<\epsilon$.

(ii) It can be proved in a similar way as the part (i), by using the fact that if $A$ $(A_N)$ is meromorphic, then $A-\lambda I$ ($A_N-\lambda I_N$) is Drazin invertible for every $\lambda\ne 0$.
\end{proof}

\begin{corollary} \label{cor1}
Let $A\in L(X)$. Then

\snoi {\rm (i)} $\sigma_{gSR}(A)$ and $\sigma_{gS\M}(A)$ are compact;

\snoi {\rm (ii)}    $\sigma_{eS}(A)\setminus\sigma_{gSR}(A)\subset {\rm iso}\, \sigma_{eS}(A)$,

\  $\sigma_{St}(A)\setminus\sigma_{gS\M}(A)\subset {\rm iso}\, \sigma_{St}(A)$.

%\snoi {\rm (iii)}  The sets $\sigma_{eS}(A)\setminus\sigma_{gSR}(A)$ and $\sigma_{St}(A)\setminus\sigma_{gS\M}(A)$
 %consist of  at most countably many points.
\end{corollary}
\begin{proof} %(i), (ii):
It follows from  Theorem \ref{closed}.% it follows that $\sigma_{St}(A)$  is closed and since $\sigma_{St}(A)\subset \sigma(A)$ we conclude that $\sigma_{St}(A)$  is compact.

%(iii): It follows from (ii).
%(ii): Suppose that $\lambda_0\in \sigma_{S}(A)\setminus\sigma_{St}(A)$. Then $T-\lambda_0 I$ is of Saphar type and according to Theorem \ref{closed} there exists $\epsilon>0$ such that $T-\lambda$ is Saphar for each $\lambda$ such that  $0<|\lambda-\lambda_0|<\epsilon$. This implies that $\lambda_0\in\iso\, \sigma_{S}(A)$. Therefore,  $\sigma_{S}(A)\setminus\sigma_{St}(A)\subset \iso\, \sigma_{S}(A)$,   which implies that $\sigma_{S}(A)\setminus\sigma_{St}(A)$ is  at most countable.
\end{proof}

\begin{corollary}\label{cor-F}  For  $A\in L(X)$ and each $*=l,r$ the following  holds:

\snoi {\rm (i)} %There are equalities
 \begin{eqnarray}
   \sigma_{gDR}^*(A)&=&\sigma_{gSR}(A)\cup\inter\, \sigma_{*}(A) \label{sljiva0} \\
   &=&\sigma_{gSR}(A)\cup\acc\, \sigma_{\B}^*(A) \label{sljiva1}
   % &=&\sigma_{gSR}(A)\cup\inter\, \sigma_{\B}^*(A) , \nonumber%\label{sljiva2}
 \end{eqnarray}

  and

  \begin{eqnarray}
   \sigma_{gD\M }^*(A)&=&\sigma_{gS\M }(A)\cup\inter\, \sigma_{*}(A) \label{sljiva0+}\\
   &=&\sigma_{gS\M}(A)\cup\acc\, \sigma_{D}^*(A) \label{sljiva1+}
  %  &=&\sigma_{gS\M}(A)\cup\inter\, \sigma_{D}^*(A) ; \nonumber%\label{sljiva2+}
 \end{eqnarray}

\snoi {\rm (ii)} \quad $\sigma_{gD\M}^*(A)\subset\sigma_{gDR}^*(A)\subset\sigma_{*}(A)$;

\smallskip

\snoi {\rm (iii)} \quad $\sigma_{gDR}^*(A)$ and $\sigma_{gD\M}^*(A)$ are compact;

\smallskip

 \snoi {\rm (iv)} \quad
$\inter\, \sigma_{gD\M}^*(A) = \inter\, \sigma_{gDR}^*(A)=\inter\, \sigma_{gD}^*(A)=\inter\, \sigma_{D}^*(A)=\inter\, \sigma_{\B}^*(A) = \inter\, \sigma_{*}(A)$;

\smallskip

\snoi {\rm (v)} \quad $\p\sigma_{gD\M}^*(A)\subset\p\, \sigma_{gDR}^*(A)\subset \p\sigma_{gD}^*(A)\subset\p\sigma_{D}^*(A)\subset\p\sigma_{\B}^*(A)\subset\p \, \sigma_{*}(A)$,

\smallskip

\snoi {\rm (vi)} \quad $\sigma_{\B}^*(A)\setminus \sigma_{gDR}^*(A)=(\iso\, \sigma_{\B}^*(A))\setminus\sigma_{gSR}(A)$,

\ \quad  $\sigma_{D}^*(A)\setminus \sigma_{gD\M}^*(A)=(\iso\, \sigma_{D}^*(A))\setminus\sigma_{gS\M}(A)$;

\smallskip

 \snoi {\rm (vii)} \quad $\sigma_{\B}^*(A)\setminus \sigma_{gDR}^*(A)$ and $\sigma_{D}^*(A)\setminus \sigma_{gD\M}^*(A)$ are  at most countable.
\end{corollary}
\begin{proof} (i): It follows from Theorem \ref{prva Riesz} and Theorem \ref{druga Riesz}.

(ii):  It is clear.

(iii): It follows from \eqref{sljiva1},
 \eqref{sljiva1+}  and Corollary \ref{cor1} (i).

(iv):
From
\eqref{sljiva0+} we have  that $\inter\,  \sigma_{*}(A)\subset \sigma_{gD\M}^*(A)$, and hence $\inter\,  \sigma_{*}(A)\subset \inter\, \sigma_{gD\M}^*(A)$. From (ii) it follows that
$ \inter\, \sigma_{gD\M}^*(A)\subset \inter\,  \sigma_{*}(A)$. Hence
\begin{equation*}%\label{iop}
 \inter\, \sigma_{gD\M}^*(A)= \inter\,  \sigma_{*}(A),
  \end{equation*}
  which together with the inclusions
\begin{equation*}%\label{ink o}
 \sigma_{gD\M}^*(A)\subset\sigma_{gDR}^*(A)\subset \sigma_{gD}^*(A)\subset\sigma_{D}^*(A)\subset\sigma_{\B}^*(A)\subset\sigma_{*}(A)
\end{equation*}
gives the desired equalities.

(v): From (ii) and (iv) it follows that
 $$\partial  \sigma_{gD\M}^*(A)\subset \sigma_{gD\M}^*(A)\setminus \inter \, \sigma_{gD\M}^*(A)\subset \sigma_{gDR}^*(A)\setminus \inter \, \sigma_{gDR}^*(A)
= \partial  \sigma_{gDR}^*(A).$$

 The rest of  inclusions can be  proved analogously.

(vi): It follows from \eqref{sljiva1} and \eqref{sljiva1+} .

(vii): It follows from (vi).
\end{proof}

\begin{corollary}\label{SVEP}
\snoi {\rm (i)} Let $A\in L(X)$ have the SVEP. Then $\acc \, \sigma_{\B}^l(A)\subset \sigma_{gSR}(A)$ and
 $\acc\, \sigma_{D}^l(A)\subset \sigma_{gS\M}(A)$.

\snoi {\rm (ii)} For $A\in L(X)$ let  $A^\prime$ have the SVEP. Then $\acc\, \sigma_{ \B}^r(A)\subset \sigma_{gSR}(A)$ and $\acc\, \sigma_{D}^r(A)\subset\sigma_{gS\M}(A)$.

\end{corollary}
\begin{proof} (i): It follows from the equivalence (ix)$\Longleftrightarrow$(x) in Theorem \ref{prva Riesz} and Theorem \ref{prva Meromorphic}.

(ii): It follows from the equivalence (ix)$\Longleftrightarrow$(x) in Theorem \ref{druga Riesz} and Theorem \ref{druga Meromorphic}.
\end{proof}

\begin{theorem}\label{partial} For  $A\in L(X)$ and  each $*=l,r$
there are inclusions
\begin{equation}\label{poss-0}
 \partial\sigma_{*}(A)\cap\acc\, \sigma_{\B}^*(A)\subset \p\sigma_{gSR}(A)
\end{equation}
and
\begin{equation}\label{poss-0+}
 \partial\sigma_{*}(A)\cap\acc\, \sigma_{D}^*(A)\subset \p\sigma_{gS\M}(A) .
\end{equation}
\end{theorem}
\begin{proof}
Suppose that $A-\lambda I$ admit a GSRD and  $\lambda\in \partial \, \sigma_{l}(A)$. Then $\lambda\notin \inter\, \sigma_{l}(A)$ and from Theorem \ref{prva Riesz} it follows   that $\lambda\notin \acc\, \sigma_{\B}^l(A)$. Hence
\begin{equation}\label{presek2}
  \partial \, \sigma_{l}(A)\cap \acc\, \sigma_{\B}^l(A)\subset \sigma_{gSR}(A).
\end{equation}
 Let $\lambda\in \partial \, \sigma_{l}(A)\cap \acc\, \sigma_{\B}^l(A)$.  Then  there exists a sequence $ (\lambda_n)$ which converges to $\lambda $ and  such that $A- \lambda_n$ is left inverible for every $n\in\NN$. Hence $A- \lambda_n$ admits a GSRD, and so $\lambda_n\notin \sigma_{gSR}(A)$ for every $n\in\NN$. As from \eqref{presek2} we have that $\lambda\in \sigma_{gSR}(A)$, it follows  that $\lambda\in \partial\, \sigma_{gSR}(A)$. This proves  the  inclusion \eqref{poss-0} for the case $*=l$.
  The rest of inclusions can be proved similarly by using Theorems \ref{druga Riesz}, \ref{prva Meromorphic} and \ref{druga Meromorphic}.
\end{proof}

If $K\subset\CC$ is compact, then
for $\lambda\in\partial K$  the following  equivalence holds:
\begin{equation}\label{zz}
 \lambda\in \acc\, K\Longleftrightarrow
\lambda\in \acc\, \partial K.
\end{equation}

\begin{corollary}\label{zad} For  $A\in L(X)$ and each $*=l,r$,
   if  $$\sigma_{*}(A)=\partial\sigma (A)\subset \acc\,\sigma (A) ,$$
    then
\begin{equation}\label{z}
  \sigma_{gS\M}(A)=\sigma_{gSR}(A)=\sigma_{gS}(A)= \sigma_{St}=\sigma_{eS}(A)=\sigma_{S}(A)=\sigma_{*}(A)
\end{equation}
and
\begin{equation}\label{z1}
  \sigma_{gD\M}^*(A)=\sigma_{gDR}^*(A)=\sigma_{*}(A).
\end{equation}

\end{corollary}
\begin{proof}
From $\sigma_{*}(A)=\partial\sigma (A)$  we get  that
 $\sigma_{*}(A)=\partial\sigma (A)\subset \partial\sigma_{*}(A)\subset\sigma_{*}(A) $, and so $\sigma_{*}(A)=\partial\sigma^* (A)$. According to  (\ref{zz})
we conclude  that every $\lambda\in \partial\sigma (A)$ is an accumulation point of
 $\partial\sigma (A)=\sigma_{*}(A)$. Thus $\sigma_{*}(A)=\partial\sigma_{*} (A)=\acc\, \sigma_{*}(A)$. From \cite[Corollary 5.5 (i)]{ZS} it follows that
  $\acc\, \sigma_{*}(A)\subset  \sigma_{D}^*(A)\subset\sigma_{*}(A)$, and so $\sigma_{*}(A)=\partial\sigma_{*}(A)=\acc\, \sigma_{D}^*(A)$.  Using
Theorem \ref{partial} we conclude  that
$$
\sigma_{*}(A)=\p\sigma_{*}(A)\cap \acc\, \sigma_{D}^*(A)\subset  \sigma_{gS\M}(A).
$$
As $\sigma_{gS\M}(A)\subset \sigma_{*}(A)$, we obtain that $\sigma_{gS\M}(A)=\sigma_{*}(A)$, which together with the inclusions
$$
\sigma_{gS\M}(A)\subset \sigma_{gSR}(A)\subset \sigma_{gS}(A)\subset  \sigma_{St}\subset \sigma_{eS}(A)\subset \sigma_{S}(A)\subset\sigma_{*}(A)
$$
 gives
\eqref{z}. Corollary \ref{cor-F} (i) and \eqref{z} imply  \eqref{z1}.
\end{proof}
\begin{theorem}\label{rubS} Let $A\in L(X)$. Then %for each $*=l,r$,
the following  inclusions hold:
%\begin{eqnarray*}
%% \nonumber to remove numbering (before each equation)
%  &&\partial\sigma_{gDR}(A)\subset\partial\sigma_{gDR}^l(A)\subset\partial\sigma_{gSR}(A),\\
% &&\partial\sigma_{gDR}(A)\subset\partial\sigma_{gDR}^r(A)\subset\partial\sigma_{gSR}(A)
%\end{eqnarray*}
%and
%\begin{eqnarray*}
%% \nonumber to remove numbering (before each equation)
%  &&\partial\sigma_{gDR\M}(A)\subset\partial\sigma_{gD\M}^l(A)\subset\partial\sigma_{gS\M}(A),\\
%&& \partial\sigma_{gD\M}(A)\subset\partial\sigma_{gDR}^r(A)\subset\partial\sigma_{gS\M}(A)
%\end{eqnarray*}

\noindent {\rm (i)}

{\footnotesize
$$\begin{array}{ccccccccccc}
&&&&  \partial\,  \sigma_{gDR}^l(A) & & \\
&&& \rotatebox{20}{$\subset$} & & \rotatebox{-20}{$\subset$}&\\
  & &  \partial\, \sigma_{gDR}(A)
   & &&  &\partial\,  \sigma_{gSR}(A)\\
&&& \rotatebox{-20}{$\subset$} & &\rotatebox{20}{$\subset$}&\\
& &&&   \partial\, \sigma_{gDR}^r(A)&  & &\\
\end{array}$$
}

and

{\footnotesize
$$\begin{array}{ccccccccccc}
&&&&  \partial\,  \sigma_{gD\M}^l(A) & & \\
&&& \rotatebox{20}{$\subset$} & & \rotatebox{-20}{$\subset$}&\\
  & &  \partial\, \sigma_{gD\M}(A)
   & &&  &\partial\,  \sigma_{gS\M}(A);\\
&&& \rotatebox{-20}{$\subset$} & &\rotatebox{20}{$\subset$}&\\
& &&&   \partial\, \sigma_{gD\M}^r(A)&  & &\\
\end{array}$$
}

%\item

\snoi {\rm (ii)} \quad $\eta\sigma_{gSR}(A)=\eta\sigma_{gDR}^l(A)=\eta\sigma_{gDR}^r(A)=\eta\sigma_{gDR}(A)$,

\smallskip

 \quad $\eta\sigma_{gS\M}(A)=\eta\sigma_{gD\M}^l(A)=\eta\sigma_{gD\M}^r(A)=\eta\sigma_{gD\M}(A)$;

\smallskip

\snoi {\rm (iii)} Let  the  complement of  $\sigma_{*}(A)$ be connected,  where $\sigma_{*}\in\{\sigma_{gSR},\sigma_{gDR}^l,\sigma_{gDR}^r \}$. Then $\sigma_{*}(A)=\sigma_{gRD}(A)$.

Let  the  complement of  $\sigma_{*}(A)$ be connected,  where $\sigma_{*}\in\{\sigma_{gS\M},\sigma_{gD\M}^l,\sigma_{gD\M}^r \}$. Then $\sigma_{*}(A)=\sigma_{gD\M}(A)$.

%\end{enumerate}
\end{theorem}
\begin{proof} (i): As

{\footnotesize
$$\begin{array}{ccccccccccc}
&&&&   \,  \sigma_{gDR}^l(A) & & \\
&&& \rotatebox{20}{$\subset$} & & \rotatebox{-20}{$\subset$}&\\
  & &   \, \sigma_{gSR}(A)
   & &&  & \,  \sigma_{gDR}(A);\\
&&& \rotatebox{-20}{$\subset$} & &\rotatebox{20}{$\subset$}&\\
& &&&    \, \sigma_{gDR}^r(A)&  & &\\
\end{array}$$
}

\noindent  according to Corollary \ref{cor1}  (i), Corollary \ref{cor-F} (iii) and  \eqref{spec.1} it is enough to prove that
$
  \partial\, \sigma_{*}(A)\subset\sigma_{gSR}(A),
$
for  $\sigma_*\in\{ \sigma_{gDR}^l, \sigma_{gDR}^r, \sigma_{gDR}   \}$. From \eqref{sljiva0} and Corollary \ref{cor-F} (iv) it follows that
\begin{eqnarray*}
  \partial\,  \sigma_{gDR}^l(A) &=& \sigma_{gDR}^l(A)\setminus \inter\, \sigma_{gDR}^l(A)=(\sigma_{gSR}(A)\cup \inter\, \sigma_{l}(A))\setminus \inter\, \sigma_{l}(A)\\&=&\sigma_{gSR}(A)\setminus \inter\, \sigma_{l}(A)
   \subset  \sigma_{gSR}(A).
\end{eqnarray*}
Similarly, $\partial\,  \sigma_{gDR}^r(A)\subset  \sigma_{gSR}(A)$.

From  Corollary \ref{cup} we conclude that  $ \sigma_{gDR}(A)= \sigma_{gSR}(A)\, \cup\, \inter\, \sigma(A)$, which implies that $ \inter\, \sigma_{gDR}(A)=  \inter\, \sigma(A)$. Applying the same method as above we conclude that  $\partial\,  \sigma_{gDR}(A)\subset  \sigma_{gSR}(A)$.
%From $\partial\,  \sigma_{gDR}(A)\subset  \sigma_{gKR}(A)$ \cite[Theorem 3.10]{ZC} and  $ \sigma_{gKR}(A)\subset  \sigma_{gSR}(A)$, it follows that $\partial\,  \sigma_{gDR}(A)\subset  \sigma_{gSR}(A)$.

The rest of the inclusions can be proved similarly.

(ii): It follows from \eqref{spec.1}, (i) and (ii).

(iii): Let the complement of  $\sigma_{*}(A)$  be  connected,  where $\sigma_{*}\in\{\sigma_{gSR},\sigma_{gDR}^l,\sigma_{gDR}^r \}$. Then by using (iii) we get  $\sigma_{gDR}(A)\supset \sigma_{*}(A)=\eta \sigma_{*}(A)=\eta \sigma_{gDR}(A)\supset \sigma_{gDR}(A)$, and so $\sigma_{*}(A)=\sigma_{gDR}(A)$.

The second assertion can be proved similarly.
\end{proof}

From Theorem \ref{rubS} (ii) it follows that
if one of the spectra $\sigma_{gSR}(A),\sigma_{gDR}^l(A),\sigma_{gDR}^r(A)$, $\sigma_{gDR}(A)$ is at most countable, then  they are equal.
Also, if one of the spectra $\sigma_{gS\M}(A)$, $\sigma_{gD\M}^l(A),\sigma_{gD\M}^r(A),\break \sigma_{gD\M}(A)$ is at most countable, then  they are equal. Hence using \cite[Theorem 3.11]{ZC} and \cite[Theorem 12]{ZD} respectively, we get the following corollaries.

\begin{corollary} For $A\in L(X)$ the following statements are equivalent:

\snoi {\rm (i)} $\sigma_{ gSR}(A)=\emptyset$;

\snoi {\rm (ii)} $\sigma_{ gDR}^l(A)=\emptyset$;

\snoi {\rm (iii)} $\sigma_{ gDR}^r(A)=\emptyset$;

\snoi {\rm (iv)} $\sigma_{ gDR}(A)=\emptyset$;

\snoi {\rm (v)} $\sigma_{ gKR}(A)=\emptyset$;

\snoi {\rm (vi)} $A$ is polynomially Riesz;

\snoi  {\rm (vii)} $\sigma_{\B}(A)$ is a finite set.

\end{corollary}
%\begin{proof} It is follows from Theorem \ref{rubS} (iii) and
%
%\end{proof}

\begin{corollary}
For  $A\in L(X)$ the following statements are equivalent:

\snoi {\rm (i)} $\sigma_{ gS\M}(A)=\emptyset$;

\snoi {\rm (ii)} $\sigma_{gD\M}^l(A)=\emptyset$;

\snoi {\rm (iii)} $\sigma_{ gD\M}^r(A)=\emptyset$;

\snoi {\rm (iv)} $\sigma_{ gD\M}(A)=\emptyset$;

\snoi {\rm (v)} $\sigma_{ gK\M}(A)=\emptyset$;

\snoi {\rm (vi)} $A$ is polynomially meromorphic;

\snoi  {\rm (vii)} $\sigma_{D}(A)$ is a finite set.

%\snoi  {\rm (v)} $\theta (A)$ is a finite set,  where $\theta$ is one of $\sigma_{TUD}$, $\sigma_{\bf B\Phi_+}$, $\sigma_{\bf B\Phi_-}$, $\sigma_{\bf B\Phi}$, $\sigma_{\bf B\W_+}$,  $\sigma_{\bf B\W_-}$,  $\sigma_{\bf B\W}$, $\sigma_{\bf B\B_+}$, $\sigma_{\bf B\B_-}$.
\end{corollary}
%\begin{proof}
%
%\end{proof}

\begin{example}\rm The $Ces\acute{a}ro\ operator$ $C_p$ on the classical Hardy space $H_p(\DDD)$, $\DDD$ is the open unit disc and $1<p<\infty$, is defined  by
$$
 (C_pf)(\lambda)=\ds\frac 1{\lambda}\int_0^{\lambda}\ds\frac{f(\mu)}{1-\mu}\, d\mu,\ \, {\rm for\ all\ }f\in H_p(\DDD)\ {\rm and\ }\lambda\in\DDD.
 $$
We recall  that its spectrum is  the closed disc $\Gamma_p$ centered at $p/2$ with radius $p/2$ and $\sigma_{Kt}(C_p)=\partial \Gamma_p$  \cite{aienarosas}.
%,  and also  $\sigma_{St}(C_p)=\sigma_{gK\M}(C_p)=\partial \Gamma_p$ \cite{ZS}, \cite{ZD}.
%\cite[Example 5.13]{ZS}, \cite[Example 1]{ZD}. $\sigma_{\Phi}(C_p)=\partial \Gamma_p$ \cite{Mill}, \cite{aienarosas}.
  %As $\sigma_{gK\M}(C_p)\subset \sigma_{gS\M}(C_p)\subset \sigma_{gSR}(C_p)\subset \sigma_{gS}(C_p)\subset \sigma_{St}(C_p)$, we get that $\sigma_{gS\M}(C_p)= \sigma_{gSR}(C_p)= \sigma_{gS}(C_p)=\partial \Gamma_p$.
   Since for $\lambda\in\inter \Gamma_p$ we have that $C_p-\lambda I$ is Fredholm,  bounded below and  not surjective \cite[Example 5.13]{ZS}, we have that %$\sigma_{\B}(C_p)=\Gamma_p$,$\sigma_{l}(C_p)=\sigma_{\B}^l(C_p)=\partial\Gamma_p$
   $\sigma_{r}(C_p)=\Gamma_p$ and  $\sigma_{l}(C_p)=\partial\Gamma_p$. From Corollary \ref{zad} it follows that $\sigma_{gS\M}(C_p)= \sigma_{gSR}(C_p)=\partial \Gamma_p$ and $\sigma_{gD\M}^l(C_p)= \sigma_{gDR}^l(C_p)=\partial \Gamma_p$,
while from
  Corollary  \ref{cor-F} (i) we have
  $
\sigma_{gDR}^r(C_p) = \sigma_{gSR}(C_p)\cup \inter\, \sigma_{r}(C_p)=\Gamma_p$ and
$\sigma_{gD\M}^r(C_p) = \sigma_{gS\M}(C_p)\cup \inter\, \sigma_{r}(C_p)=\Gamma_p$.

%\begin{eqnarray*}
%% \nonumber to remove numbering (before each equation)
%\sigma_{gDR}^r(C_p) &=& \sigma_{gSR}(C_p)\cup \inter\, \sigma_{r}(C_p)=\Gamma_p,\\
%\sigma_{gD\M}^r(C_p) &=& \sigma_{gS\M}(C_p)\cup \inter\, \sigma_{r}(C_p)=\Gamma_p.
%   \end{eqnarray*}
\end{example}

\medskip \noindent {\bf Acknowledgement.} The author wishes to thank Professor  Mohammed Berkani for helpful conversations concerning the paper.

\medskip

\noindent

\noindent{Sne\v zana
\v C. \v Zivkovi\'c-Zlatanovi\'c}

\noindent{University of Ni\v s\\
Faculty of Sciences and Mathematics\\
P.O. Box 224, 18000 Ni\v s, Serbia}

\noindent {\it E-mail}: {\tt mladvlad@mts.rs}, {\tt snezana.zivkovic-zlatanovic@pmf.edu.rs}

\end{document}